\documentclass[12pt]{article}

\usepackage{amssymb,amsmath,amsthm,amsfonts,amscd,bbm}
\usepackage{enumitem}
\usepackage{pdflscape}
\usepackage{caption}
\usepackage{bm}

\usepackage{ifpdf}
\ifpdf
\usepackage[pdftex]{graphicx}
\else
\usepackage[dvips]{graphicx}
\fi
\usepackage[all]{xy}
\usepackage{hyperref}
\usepackage{xcolor}
\usepackage{graphicx}
\usepackage[outdir=./]{epstopdf}
\usepackage{enumitem}
\usepackage{tensor}
\usepackage{tikz-cd}
\usepackage{multicol}
\usepackage{mathtools}
\usepackage{tocvsec2}
\usepackage{bbm}
\usepackage{longtable}
\usepackage{fullpage}

\newcommand{\doi}[1]{\href{https://dx.doi.org/#1}{{\small DOI:#1}}}

\newcommand{\googlebooks}[1]{(preview at \href{https://books.google.com/books?id=#1}{google books})}

\newcommand{\numdam}[1]{}

\usepackage{mathrsfs}
\DeclareMathAlphabet{\mathpzc}{OT1}{pzc}{m}{it}

\def\semicolon{;}
\def\applytolist#1{
    \expandafter\def\csname multi#1\endcsname##1{
        \def\multiack{##1}\ifx\multiack\semicolon
            \def\next{\relax}
        \else
            \csname #1\endcsname{##1}
            \def\next{\csname multi#1\endcsname}
        \fi
        \next}
    \csname multi#1\endcsname}

\def\calc#1{\expandafter\def\csname c#1\endcsname{{\mathcal #1}}}
\applytolist{calc}QWERTYUIOPLKJHGFDSAZXCVBNM;
\def\bbc#1{\expandafter\def\csname bb#1\endcsname{{\mathbb #1}}}
\applytolist{bbc}QWERTYUIOPLKJHGFDSAZXCVBNM;
\def\bfc#1{\expandafter\def\csname bf#1\endcsname{{\mathbf #1}}}
\applytolist{bfc}QWERTYUIOPLKJHGFDSAZXCVBNM;
\def\sfc#1{\expandafter\def\csname s#1\endcsname{{\sf #1}}}
\applytolist{sfc}QWERTYUIOPLKJHGFDSAZXCVBNM;
\def\fc#1{\expandafter\def\csname f#1\endcsname{{\mathfrak #1}}}
\applytolist{fc}QWERTYUIOPLKJHGFDSAZXCVBNM;

\usepackage{tikz}
\usepackage{tikz-cd}

\makeatletter
\def\fixtikzforbreqn#1#2{%
  \protected\edef#1{\noexpand\ifmmode\mathchar\the\mathcode`#2 \noexpand\else#2\noexpand\fi}%
}
\fixtikzforbreqn\tikz@nonactivesemicolon;
\fixtikzforbreqn\tikz@nonactivecolon:
\fixtikzforbreqn\tikz@nonactivebar|
\fixtikzforbreqn\tikz@nonactiveexlmark!
\makeatother

\usepackage{breqn}   

\usetikzlibrary{arrows,backgrounds,patterns.meta}
\usetikzlibrary{positioning,shadings,cd}
\usetikzlibrary{shapes}
\usetikzlibrary{backgrounds}
\usetikzlibrary{decorations,decorations.pathreplacing,decorations.markings,decorations.pathmorphing}
\usetikzlibrary{fit,calc,through}
\usetikzlibrary{external}
\usetikzlibrary{arrows}
\tikzset{vertex/.style = {shape=circle,draw,fill=black,inner sep=0pt,minimum size=5pt}}
\tikzset{edge/.style = {->,> = latex', bend right}}
\tikzset{
	super thick/.style={line width=3pt}
}
\tikzset{
    quadruple/.style args={[#1] in [#2] in [#3] in [#4]}{
        #1,preaction={preaction={preaction={draw,#4},draw,#3}, draw,#2}
    }
}
\tikzstyle{shaded}=[fill=red!10!blue!20!gray!30!white]
\tikzstyle{unshaded}=[fill=white]
\tikzstyle{empty box}=[circle, draw, thick, fill=white, opaque, inner sep=2mm]
\tikzstyle{annular}=[scale=.7, inner sep=1mm, baseline]
\tikzstyle{rectangular}=[scale=.75, inner sep=1mm, baseline=-.1cm]
\tikzstyle{mid>}=[decoration={markings, mark=at position 0.5 with {\arrow{>}}}, postaction={decorate}]
\tikzstyle{mid<}=[decoration={markings, mark=at position 0.5 with {\arrow{<}}}, postaction={decorate}]
\tikzstyle{over}=[double, draw=white, super thick, double=]
\tikzstyle{snake}=[decorate, decoration={snake, segment length=1mm, amplitude=.3mm}]
\tikzstyle{saw}=[decorate, decoration={saw, segment length=.7mm, amplitude=.25mm}]

\tikzstyle{coupon}=[draw, very thick, rectangle, rounded corners=5pt]
\tikzset{Rightarrow/.style={double equal sign distance,>={Implies},->},
triplecd/.style={-,preaction={draw,Rightarrow}},
quadruplecd/.style={preaction={draw,Rightarrow,
shorten >=0pt
},
shorten >=1pt,
-,double,double
distance=0.2pt}}
\tikzset{
    tripleline/.style args={[#1] in [#2] in [#3]}{
        #1,preaction={preaction={draw,#3},draw,#2}
    }
}
\tikzstyle{triple}=[tripleline={[line width=.15mm,black] in
      [line width=.7mm,white] in
      [line width=1mm,black]}] 
\tikzset{
    quadrupleline/.style args={[#1] in [#2] in [#3] in [#4]}{
        #1,preaction={preaction={preaction={draw,#4},draw,#3}, draw,#2}
    }
}
\tikzstyle{quadruple}=[quadrupleline={[line width=.3mm,white] in
      [line width=.6mm,black] in
      [line width=1.2mm,white] in
      [line width=1.5mm,black]}]

\theoremstyle{plain}
\newtheorem{thm}{Theorem}[section]
\newtheorem*{thm*}{Theorem}

\newtheorem{cor}[thm]{Corollary}

\newtheorem*{cor*}{Corollary}

\newtheorem*{conj*}{Conjecture}
\newtheorem{lem}[thm]{Lemma}
\newtheorem*{lem*}{Lemma}

\newtheorem{prop}[thm]{Proposition}
\newtheorem{quest}[thm]{Question}

\newtheorem*{quest*}{Question}
\newtheorem*{claim*}{Claim}

\theoremstyle{definition}
\newtheorem{defn}[thm]{Definition}

\newtheorem{ex}[thm]{Example}
\newtheorem{sub-ex}[thm]{Sub-Example}
\newtheorem{counter-ex}[thm]{Counter-Example}
\newtheorem{rem}[thm]{Remark}
\newtheorem*{rem*}{Remark}
\newtheorem{remark}[thm]{Remark}

\usepackage{xcolor}
\definecolor{dark-red}{rgb}{0.7,0.25,0.25}
\definecolor{dark-blue}{rgb}{0.15,0.15,0.55}
\definecolor{medium-blue}{rgb}{0,0,.8}
\definecolor{DarkGreen}{RGB}{0,150,0}
\definecolor{rho}{named}{red}
\hypersetup{
   colorlinks, linkcolor={purple},
   citecolor={medium-blue}, urlcolor={medium-blue}
}

\def\altdb{\vadjust{\vbox to 0pt{\vss\hbox{\kern \hsize
\quad{\dbend}}\kern\baselineskip\kern-10pt}}}

\newcommand{\noshow}[1]{}

\begin{document} 

\title{Centralizers of discrete Temperley-Lieb-Jones subfactors}
\author{Corey Jones, Emily McGovern}

\maketitle

\begin{abstract}
Discrete, unimodular inclusions of factors $(N\subseteq M, E)$ with $N$ of type $\rm{II}_{1}$ have a natural notion of standard invariant, generalizing the finite index case. When the unitary tensor category of $N$-$N$ bimodules generated by $_{N}L^{2}(M, \tau\circ E)_{N}$ is equivalent to the Temperley-Lieb-Jones category $\text{TLJ}(\delta)$, the associated discrete standard invariants are classified in terms of fair and balanced $\delta$-graphs. Many examples of these subfactors naturally arise in the context of the Guionnet-Jones-Shlyakhtenko (GJS) construction for graphs. In this paper, we compute the discrete standard invariant of the centralizer subfactor $N\subseteq M^{\phi}$ for the canonical state $\phi=\tau\circ E$, which is again a discrete subfactor of $\text{TLJ}(\delta)$-type. We show that the associated fair and balanced $\delta$-graph behaves analogously to a universal covering space of the original fair and balanced $\delta$-graph. As an application, we obtain an obstruction to the realization of discrete tracial TLJ-type standard invariants by subfactors of a $\rm{II}_{1}$ factor $M$ in terms of the fundamental group of M.

\end{abstract}

\tableofcontents
\section{Introduction}
The modern theory of finite index subfactors emerged from Vaughan Jones' surprising discovery that the set of possible indices for inclusions of $\rm{II}_{1}$ factors has a discrete and continuous part \cite{MR0696688}. This led to the discovery of an extremely rich algebraic structure associated with a finite index subfactor called its standard invariant \cite{MR1334479, MR4374438, MR996454, MR1642584, MR1257245, MR1966524}, which underlies deep connections between subfactors and category theory, topology, and quantum physics. 

 Moving beyond the finite index setting, a natural class of subfactors to consider are the \textit{discrete subfactors} \cite{MR1622812}, which generalize the inclusion of a factor into the crossed product by an outer action of a discrete group. In \cite{MR3948170}, a notion of standard invariant was defined for discrete, irreducible, unimodular inclusions $(N\subseteq M, E)$ in the case where $N$ is $\rm{II}_{1}$. This can be described as a pair $(\mathcal{C}, \textbf{M})$, where $\mathcal{C}$ is the unitary tensor category of bifinite $N$-$N$ bimodules generated by $L^{2}(M)$ and $\textbf{M}$ is a connected W*-algebra object internal to $\mathcal{C}$ in the sense of \cite{MR3687214}. In this setting, $M$ can be type $\rm{III}$, in which case there is a canonical intermediate subfactor $N\subseteq M^{\phi}$, where $M^{\phi}$ is the centralizer of the state $\phi=\tau\circ E$ on $M$. 
 
 In this paper, our goal is to compute the standard invariant of the centralizer $N\subseteq M^{\phi}$ in terms of the standard invariant of $N\subseteq M$ in the case of discrete subfactors of \textit{Temperley-Lieb-Jones} (TLJ) type. These are the discrete subfactor whose unitary tensor category of bifinite $N$-$N$ bimodules generated by $L^{2}(M)$ is equivalent to the (even part) of the finite index $\text{TLJ}(\delta)$ subfactor standard invariant, or its unoriented tensor categorical extension (equivalent to $\text{Rep}(SU_{-q}(2))$), for some $\delta=q+q^{-1}\ge 2$. By \cite{MR3420332}, the possible connected W*-algebra objects $\textbf{M}$ are parameterized by fair and balanced $\delta$-graphs. These consist of pairs $(\Gamma,w,*)$, where $\Gamma$ is a locally finite, directed graph (possibly with multiple edges between vertices), $w:E(\Gamma)\rightarrow \mathbbm{R}^{\times}_{>0}$ is a function called the weight function, and $*\in V(\Gamma)$ is a distinguished vertex. This data is required to satisfy certain conditions, most notably that for any vertex $v$, the sum of the weights of the edges that end at $v$ must be $\delta$. If the edge weighting is induced by a vertex weighting, the graph is called tracial, and corresponds to the case that the super factor is also type $\rm{II}_{1}$.

 Examples of discrete subfactors of TLJ type arise from the the GJS construction for graphs \cite{MR2732052, MR2807103, MR2645882, MR3110503, HartglassNelson2020}. Associated to a fair and balanced $\delta$-graph, we take the (possibly non-tracial) version of the GJS graph algebra $M(\Gamma)$, cut down at the vertex projection associated to our distinguished vertex $p_{*}$. This contains a copy of the Temperley-Lieb-Jones GJS subalgebra at loop parameter $\delta$, $\text{Gr}(\mathcal{TL}_{\delta})$. Then $\text{Gr}(\mathcal{TL}_{\delta})\subseteq p_{*}M(\Gamma)p_{*}$ is discrete subfactor of $\text{TLJ}(\delta)$-type, whose standard invariant is precisely the fair and balanced $\delta$-graph used in the construction (c.f. \cite[Section 6.2]{MR3948170}). In the case $\delta\ge 2$, $T_{\delta}\cong LF_{\infty}$ \cite{MR3110503}, and if $\Gamma$ is tracial, $p_{*}M(\Gamma)p_{*}\cong LF_{t}$ for some $t\in (0,\infty]$ \cite{MR2732052, MR2807103, MR3110503}, while if $\Gamma$ is not tracial, we obtain a type III factor which is shown to be a free Araki-Woods factor in the case where $\Gamma$ is finite \cite{HartglassNelson2020}.
 
Given a TLJ type standard invariant described by a fair and balanced $\delta$-graph $\Gamma$, the centralizer subfactor should be described by another fair and balanced $\delta$-graph, but one which is always tracial (meaning the associated super factor is again type $\rm{II}_{1}$). In this paper, we provide a combinatorial description of this new graph, which we denote by $\Gamma_{tr}$ (see Definition \ref{tracial cover def}). Interestingly, we find that $\Gamma_{tr}$ behaves very much like a universal covering space of the original graph $\Gamma$, where the analog of ``homotopy of paths" is controlled by the edge weights on the graph. We give explicit calculations in several examples.

Motivated by the analogy with covering space theory, we also address the converse question: give a tracial fair and balanced $\delta$-graph $\Gamma$, what are the non-tracial fair and balanced $\delta$-graphs $\Lambda$ such $\Gamma\cong \Lambda_{tr}$? We show that such $\Lambda$ are always ``quotients" of $\Gamma$ by an action of a group. We utilize this in our application to the realization problem for discrete TLJ standard invariants, described in the next section.

\subsection{Application to the realization problem}

One natural question that arises concerning the relationship between abstract standard invariants and subfactors is, given an abstract standard invariant and a $\rm{II}_{1}$ factor $M$, is there a subfactor $N\subseteq M$ realizing this standard invariant, and if so how many? We call this type of question a \textit{realization problem}.

We briefly review the main results on realization problems in the finite index case. Building on earlier work in the group case \cite{MR448101,MR587749,MR596082}, the first major generic realization results are due to Popa \cite{MR1278111, MR1339767}, which completely solve the realization problem in an important case: all strongly amenable standard invariants are \textit{uniquely realized} by subfactors of the hyperfinite $\rm{II}_{1}$ factor $R$. This reduces the problem of classifying all strongly amenable hyperfinite subfactors to the algebraic problem of classifying strongly amenable standard invariants, cementing the fundamental role of standard invariants in subfactor theory. There are also many non-amenable standard invariants that are known to be realizable in the hyperfinite $\rm{II}_{1}$ factor, but to date there is no known systematic characterization (see \cite{bisch2025newhyperfinitesubfactorsinfinite} for recent progress).  

Outside the hyperfinite case, Popa showed that all abstract standard invariants are realized by subfactors of \textit{some} $\rm{II}_{1}$ factor \cite{MR1334479} and Popa and Shlyakthenko sharpened this to show that every standard invariant is realized by a subfactor of the group von Neumann algebra $LF_{\infty}$ \cite{MR2051399}. Guillonet, Jones and Shlyakthenko have shown that finite depth standard invariants can be realized in certain interpolated free group factors \cite{MR2732052} (see also \cite{MR3110503}). In yet another direction, it is possible to give constructions of $\rm{II}_{1}$ factors whose category of bifinite bimodules is explicitly computable \cite{MR2386109, MR2471930, MR2658193, MR2504433,MR2838524}. One important consequence of this direction of inquiry is the result of \cite{MR3028581}, which shows that any abstract unitary fusion category is equivalent to the whole category of bifinite bimodules of some $\rm{II}_{1}$ factor, and in these cases this  allows for a complete characterization of subfactors of these in terms of $Q$-systems in $\mathcal{C}$ .

The above results summarize the extent of what is known concerning realization problems for finite index subfactors. One of the many difficulties with the general realization problem is that there are a lack of good obstructions for realizing abstract standard invariants in a given $\rm{II}_{1}$ factor. For example, one might hope that for a standard invariant to be realized by a subfactor of $M$, some properties and/or invariants associated to the standard invariant and to $M$ must be compatible in some way. However, at this stage there do not appear to be any obvious candidates for this kind of obstruction.

Looking to the discrete case, considerably less is known. Aside from the case of group-like unitary tensor categories $\text{Hilb}(\Gamma,\omega)$ (which can be reformulated in terms of group actions and $\Gamma$-kernels),  there are very few realization results, beyond those that can be deduced from the finite index case. For example, it follows from the results of \cite{MR2807103, MR3405915} and \cite{MR3948170}, every discrete standard invariant over $\text{TLJ}(\delta)$ can be realized as a subfactor $N\subseteq M$, where $N\cong \text{LF}_{\infty}$ and $M\cong \text{LF}_{t}$ for some $t\in (0,\infty]$. It is also known that certain other classes of discrete standard invariants, for example those related to the representation theory of compact groups, can be realized inside the hyperfinite $\rm{II}_{1}$ factor $R$ \cite{MR2409162}.

In this paper, we consider an obstruction to realizing \textit{discrete} standard invariants related to the fundamental group of a $\rm{II}_{1}$ factor. Recall that if $M$ is a $\rm{II}_{1}$ factor, the fundamental group $\mathcal{F}(M)=\{t\in \mathbbm{R}^{\times}_{>0}\ : M\cong M^{t}\}$. This numerical invariant plays an important role in the modern theory of $\rm{II}_{1}$ factors, but calculating it is notoriously difficult. Nevertheless, it can be computed explicitly in many examples using Popa's deformation rgidity techniques (c.f. \cite{MR2215135} or \cite{ClaireSorinII_1} for an overview). 

Associated to an abstract discrete standard invariant $(\mathcal{C}, \textbf{M})$ we define a group $T_{0}(\textbf{M})\subseteq \mathbbm{R}^{\times}_{+}$, which we think of as a kind of fundamental group for the standard invariant (see Definition \ref{defineinvariant} and Remark \ref{fundgroupex}). We prove the following theorem.

\begin{thm} Let $(\mathcal{C},\textbf{M})$ be a discrete subfactor standard invariant, $M$ a $\rm{II}_{1}$ factor. Then if there exists a discrete, unimodular subfactor $N\subseteq M$ with $\text{StdInv}(N\subseteq M)\cong \mathcal{C}, \textbf{M})$, then $T_{0}(\textbf{M})\subseteq \mathcal{F}(M)$.
\end{thm}

Thus if we can compute $T_{0}(\textbf{M})$ and it is non-trivial, this can lead to an obstruction for realizing discrete standard invariants. Our results on calculating centralizers described in the previous section allows us to prove the following.

\begin{thm}(c.f Proposition \ref{prop: T_0= Wx})
For a tracial fair and balanced $\delta$-graph $\Gamma$, $T_{0}(\Gamma) = W^{\times}(\Gamma)$, where 
$$W^{\times}(\Gamma) = \{ \lambda_\alpha = w(\alpha(*)) | \alpha \in Aut((\Gamma, w)).$$ 
\end{thm}

\bigskip

\noindent For example, let $\delta=q+q^{-1}>2$. Consider the fair and balanced $\delta$-graph $A^{\delta}_{-\infty, \infty}$:

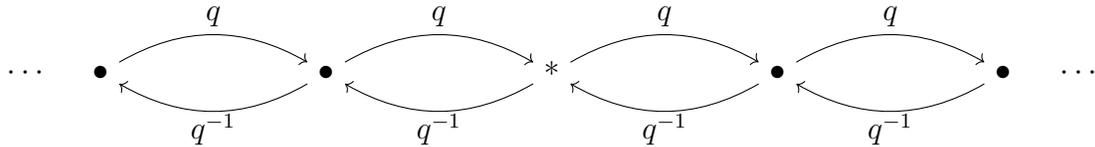
     $$\begin{tikzpicture}
        \node(none) (s) at (0,0){*} ;
        \node(none) (v)at (3,0){$\bullet$};
        \node(none) (w) at (-3,0){$\bullet$};
        \node(none) (x)at (6,0){$\bullet$};
        \node(none) (y) at (-6,0){$\bullet$};
        \node(none) (z) at (-7,0){$\cdots$};
        
        \node(none) (z) at (7,0){$\cdots$};
        \node(none) (1) at (-1.5,0.75){$q$};
        \node(none) (3) at (1.5,0.75){$q $};
         \node(none) (1) at (-4.5,0.75){$q$};
       `\node(none) (3) at (4.5,0.75){$q$};
        \node(none) (1) at (-1.5,-0.75){$q^{-1}$};
        \node(none) (3) at (1.5,-0.75){$q^{-1} $};
         \node(none) (1) at (-4.5,-0.75){$q^{-1}$};
        \node(none) (3) at (4.5,-0.75){$q^{-1} $};

         \draw [->] (w) to [out=30,in=150] (s) ;
        \draw [->] (s) to [out=210,in=330] (w) ;

        \draw [->] (s) to [out=30,in=150] (v) ;
        \draw [->] (v) to [out=210,in=330] (s) ;

        \draw [->] (v) to [out=30,in=150] (x) ;
        \draw [->] (x) to [out=210,in=330] (v) ;

         \draw [->] (y) to [out=30,in=150] (w) ;
        \draw [->] (w) to [out=210,in=330] (y) ;
    \end{tikzpicture}$$

If we identify the vertices with $\mathbbm{Z}$ with the distinguished vertex $*$ identified with $0$, then this edge weighting arises from a vertex weighting that weights the vertex $m$ by $q^{m}$. By the above theorem, we have 
$$T_{0}(A^{\delta}_{-\infty,\infty})=\{q^{n}\ :\ n\in \mathbbm{Z}\}\subseteq \mathbbm{R}^{\times}_{>0}$$

Then we have the following .

\begin{cor}
  Let $M$ be a $\rm{II}_{1}$ factor. Then if $q\notin \mathcal{F}(M)$, there is no discrete, unimodular subfactor $N\subseteq M$ with $\text{StdInv}(N\subseteq M)\cong A^{q+q^{-1}}_{-\infty,\infty}$.
\end{cor}

Recall that notorious free group factor problem, which asks whether finite (interpolated) free group factors are isomorphic for free groups of different rank. It is well known that either all interpolated free group factors are isomorphic, or they are pairwise non-isomorphic, and the first statement is equivalent to any finite interpolated free group factor having a non-trivial element in its fundamental group \cite{Rad94}. By our results above, this leads to the following corollary.

\begin{cor}
If there exists some $0<t<\infty$ and a discrete, unimodular subfactor $N\subseteq LF_{t}$ with $\text{StdInv}(N\subseteq M)\cong A^{\delta}_{-\infty,\infty}$ for some $\delta>2$, then $LF_{2}\cong LF_{3}$.
\end{cor}

In fact, it is easy to see that the realization problem for this standard invariant is actually an equivalence. In partiulcar, if all free group factors are isomorphic, they are isomorphic to $LF_{\infty}$, and $A^{\delta}_{-\infty,\infty}$ can be realized in $LF_{\infty}$.

Note that there was nothing really special about $A^{\delta}_{-\infty,\infty}$ in this statement. We could have used any standard invariant of TLJ type with non-trivial $T_{0}$. Interestingly, it is known that if the fair and balanced $\delta$- graph $\Gamma$ is finite it is neccessarily tracial, and the inclusion of the $\text{TLJ}(\delta)$ GJS algebra into the GJS algebra associated to the tracial construction gives a realization of this standard invariant in $LF_{t}$ for some finite $t$ \cite{MR2732052}. However, these discrete standard invariants have trivial $T_{0}$, see Proposition \ref{finite graph prop}.

\subsection{Acknowledgements.}

The authors would like to thank Brent Nelson for helpful comments. CJ was supported by NSF DMS 2247202. EM was supported by NSF DMS 2039316.

\section{Discrete subfactors and their standard invariants} 
\label{discrete subfactors}

\begin{defn} Consider a pair $(N\subseteq M, E)$, where $N$ is a $\rm{II}_{1}$ factor, $M$ is a factor, and $E$ is a normal conditional expectation. The inclusion is

\begin{enumerate}
    \item 
    \textit{irreducible} if $N^{\prime}\cap M=\mathbbm{C}1$.
    \item 
    \textit{discrete} if $_{N}L^{2}(M, \tau\circ E)_{N}\cong \bigoplus_{i} H_{i}$, where the $H_{i}$ are bifinite $N$-$N$ bimodules.
    \item 
    \textit{unimodular} if the bimodules $H_i$ in the above decomposition have equal left and right Murray-von Neumann dimension.
\end{enumerate}

\end{defn}

\begin{rem}
We only consider irreducible inclusions. In the above definition, if $M$ is also a $\rm{II}_{1}$ factor, then $E:M\rightarrow N$ is the unique trace-preserving conditional expectation.
\end{rem}

\subsection{Abstract standard invariants}
Recall that a unitary tensor category is a rigid C*-tensor category with simple unit object. If $\mathcal{C}$ is a unitary tensor category, a cyclic W*-module category consists of a pair $(\mathcal{M},m)$ where $\mathcal{M}$ is a W*-module category for $\mathcal{C}$, and $m\in \mathcal{M}$ is an object such that the full, replete module subcategory generated by $\mathcal{C}m$ is $\mathcal{M}$ (see \cite[Section 3.1]{MR3687214}).

The data of a cyclic W*-module category is equivalent to the data of a W*-algebra object $\textbf{M}\in \text{Vec}(\mathcal{C})$, defined via the internal end construction with respect to the object $m$. 

In general a $*$-algebra object $\textbf{M}\in \text{Vec}(\mathcal{C})$ can be viewed as a lax monoidal functor $\textbf{M}:\mathcal{C}^{op}\rightarrow \text{Vec}$, together with a conjugate linear involution satisfying some coherences. To unpack this data, let $\mathcal{C}$ be a unitary tensor category, equipped with it's standard unitary spherical structure. We have preselected a unitary dual functor representing this structure $a\mapsto \overline{a}$ in the sense of \cite{MR4133163}. We let $\alpha_{a,b,c}:(a\otimes b)\otimes c\cong a\otimes (b\otimes c)$ be the associator isomorphism. Furthermore, for $f\in \mathcal{C}(a,b)$, we set $\overline{f}=\widetilde{f}^{*}\in \mathcal{C}(\overline{a},\overline{b})$ (see \cite[Section 3.3]{MR3687214} for details).

We have that a $*$-algebra $\textbf{M}$ in $\text{Vec}(\mathcal{C})$ consists of the following data:

\begin{enumerate}
\item 
(A functor $\textbf{M}:\mathcal{C}^{op}\rightarrow \text{Vec}$): For each object $a\in \mathcal{C}$, we have a vector space $\textbf{M}(a)$
For $f\in \mathcal{C}(a,b)$, $\textbf{M}(f):\textbf{M}(b)\rightarrow \textbf{M}(a)$, such that $\textbf{M}(f\circ g)=\textbf{M}(g)\circ \textbf{M}(f)$ and $\textbf{M}(1_{a})=\text{id}_{\textbf{M}(a)}$.
\item 
(Lax monoidal structure): linear maps  $m_{a,b}:\textbf{M}(a)\otimes \textbf{M}(b)\rightarrow \textbf{M}(a\otimes b)$, which are natural in $a$ and $b$, satisfying  $$ m_{a\otimes b, c}\circ (m_{a,b}\otimes \text{id}_{\textbf{M}(c)})=\textbf{M}(\alpha_{a,b,c})\circ m_{a,b\otimes c}\circ (\text{id}_{\textbf{M}(a)}\otimes m_{b,c}).$$
\item 
(*-structure): a family conjugate linear maps $j_{a}: \textbf{M}(a)\rightarrow \textbf{M}(\overline{a})$ such that

\begin{enumerate}
\item 
(Conjugate naturality). For $f\in \mathcal{C}(a,b)$, $j_{b}\circ \textbf{M}(f)= \textbf{M}(\overline{f}) \circ j_{a}$ 
\item 
(Involutive). $j_{\overline{a}}\circ j_{a}=\text{id}_{\textbf{M}(a)}$.
\item 
(Monoidal). for $u\in \textbf{M}(a)$, $v\in \textbf{M}({b})$ $j_{a\otimes b}\circ m_{a,b}(u\otimes v) = m_{\overline{b},\overline{a}}\circ (j_{b}(v)\otimes j_{a}(u))$
\end{enumerate}
\end{enumerate}

\begin{rem} In the above definition we have suppressed the canonical unitary isomorphisms $\overline{\overline{a}}\cong a$ and $\overline{a\otimes b}\cong \overline{b}\otimes \overline{a}$ associated to the unitary dual functor.

Associated to a $*$-algebra object is a cyclic $\mathcal{C}$-module $*$- category, the category of \textit{finitely generated projective modules of $\textbf{M}$ internal to $\mathcal{C}$} (see \cite[Theorem 3.20]{MR3687214} for details). A $*$-algebra object is a \textit{W*-algebra object} if the associated cyclic $\mathcal{C}$-module $*$-category of right projective $\textbf{M}$-modules is in fact a W*-category. This is a \textit{property} of a $*$-algebra and not additional structure. For an ``internal" description of the W*-property that does not directly use the category of modules, see \cite[Prop 3.25]{MR3687214}. 

Furthermore, this construction assigns, to every W*-algebra object $\textbf{M}$ in $\text{Vec}(\mathcal{C})$ a cyclic $\mathcal{C}$-module category of projective $\textbf{M}$ modules, with $\textbf{M}$ as the distinguished generating object. Conversely, given a W*-module $\mathcal{C}$-module category $\mathcal{M}$ with distinguished generating object $m\in \mathcal{M}$, define $\textbf{M}(a):=\mathcal{M}(a\otimes m, m)$. This assembles into a W*-algebra object in the sense described above (see  \cite{MR3687214} for details). 
\end{rem}

We can apply this general framework to obtain a version of the standard invariant for finite index subfactors that applies to the infinite index discrete case. First we have the following definition obtain a \textit{connected} $W^{*}$-algebra object $\textbf{M}$. 

\begin{defn}
    A \textit{abstract discrete standard invariant} is a pair $(\mathcal{C}, \textbf{M})$, where $\mathcal{C}$ is a unitary tensor category and $\textbf{M}$ is a tensor generating, connected W*-algebra object.
\end{defn}

We note that for $\mathcal{C}\cong \text{Rep}(\mathbbm{G})$ for $\mathbbm{G}$ a compact quantum group, a discrete standard invariant is precisely the same data as an ergodic quantum homogenous space for $\mathbbm{G}$ (\cite[Theorem 6.4]{MR3121622}). As explained above, a pair $(\mathcal{C}, \textbf{M})$ is equivalently described by a triple $(\mathcal{C}, (\mathcal{M},m))$, which will frequently be useful in extracting standard invariants ``from nature".

\subsection{Concrete standard invariants}
Let $(N\subseteq M, E)$ be a discrete, irreducible inclusion of factors, where $N$ is type $\rm{II}_{1}$. Set $\phi:=\tau\circ E$, and define $\mathcal{C}:=\langle L^{2}(M, \phi)\rangle\subset \text{Bim}(N)$ the unitary tensor category of $N$-$N$ bimodule generated by $L^{2}(M, \phi)$. 

Then the irreducibility assumption implies the $N$-$M$  bimodule $_{N} L^{2}(M, \phi)_{M}$ is irreducible. We define the subcategory of $N$-$M$ bimodules $\mathcal{M}:=\langle \mathcal{C}\boxtimes_{N} L^{2}(M,\phi)_{M} \rangle$, which by construction is a cyclic $\mathcal{C}$ module category with distinguished object $ _{N} L^{2}(M,\phi)_{M}$. This category is a semisimple W* $\mathcal{C}$-module category (this follows from \cite[Corollary 4.5]{MR3687214}). We denote the associated connected W*-algebra object $\textbf{M}$. 

Then as explained above (see also \cite{MR3948170}), the associated W*-algebra object in $\text{Vec}(\mathcal{C})$ is given by 

$$\textbf{M}(a):=\text{Hom}_{N-M}(a\boxtimes_{N} L^{2}(M,\phi), L^{2}(M,\phi))$$

\noindent We are now in a position to define

$$\textbf{StdInv}(N\subseteq M):=(\mathcal{C}, \textbf{M})$$

\noindent 

We note that unlike in the finite index case, we make no explicit mention of the $M$-$M$ bimodules (or ``dual even part") in our formulation of the standard invariant. The reason is two-fold: on the one hand, the dual even part is determined by our data, and can be explicitly computed using a generalized tube algebra \cite{MR3801484, MR1966524}. On the other, for discrete subfactors the dual category is typically very difficult to work with using the standard tools of tensor category theory, and is usually non-semisimple and non-rigid. An illuminating example is the symmetric enveloping subfactor introduced by Popa \cite{MR1302385}. In this case, the dual category is the Drinfeld center $\mathcal{Z}(\text{Ind}(\mathcal{C}))$ and is generically extremely complicated \cite{MR3406647, MR3509018, MR3447719}. Even in relatively simple situations, the tensor product of objects is difficult to decompose.

Another important part of this story is the \textit{realization functor}. Given a connected W*-algebra object $\textbf{M}\in \text{Vec}(\mathcal{C})$, where $\mathcal{C}$ is a full tensor subcategory of bifinite, spherical bimodules the $\rm{II}_{1}$ factor $N$, then the realization $|\textbf{M}|$ is an irreducible, unimodular discrete extension of $N$ \cite{MR3948170}. The standard invariant gives a bijective correspondence between (isomorphism classes of) connected W*-algebra objects in $\mathcal{C}$ and discrete irreducible inclusions supported on $\mathcal{C}$. For a fixed $\mathcal{C}$, this result (and in particular the functorality of this correspondence) allows us to pass freely between the categorical world and the analytic world of discrete extensions.

For a finite index inclusion, the algebra object $\textbf{M}$ is \textit{compact} (i.e. dualizable) and may be equivalently reformulated as a Q-system (\cite{MR4079745}), which aligns with the standard categorical approach to finite index subfactors (\cite{Longo89, MR1966524, MR3308880, MR4419534}).

\begin{quest}(The standard invariant realization problem). Given a (tracial) standard invariant $(\mathcal{C}, \textbf{M})$ and a $\rm{II}_{1}$ factor $M$, does there exist an irreducible subfactor $N\subseteq M$ such that $\text{StdInv}(N\subseteq M)\cong (\mathcal{C}, \textbf{M})$? 
\end{quest}

Of course, tracial is not essential for this question but in this paper, we are mostly interested in the case of $\rm{II}_{1}$ factors so we restrict to the tracial case for this question.

To date, this question has only been directly addressed in the finite index case, as described in the introduction. Both constructing an explicit subfactor realizing a given standard invariant, or showing no realization exists, is very difficult in general. One might hope that if such a realization exists, features of a $\rm{II}_{1}$ factor (e.g. analytic properties or algebraic invariants) must be compatible with features of the standard invariant in question, since this would at least give some traction for \textit{obstructing} the realization of a given standard invariant in a given $\rm{II}_1$ factor. So far phenomena of this nature have been elusive.

One of the goals of this paper is to provide precisely such an obstruction for the infinite index, discrete case, which we carry out in the next section.

\subsection{Fundamental group invariant}
\label{modular operators}
For any connected W*-algebra object $\textbf{M}\in \text{Vec}(\mathcal{C})$, there are two canonically associated Hilbert space structures on all the finite dimensional spaces $\textbf{M}(a)$:

$$_{a}\langle f , g\rangle:=\textbf{M}(\text{coev}_{a})\left(m_{a,\overline{a}}(f\otimes j_{a}(g))\right)$$

and

$$\langle g | f\rangle_{a}:=\textbf{M}(\text{ev}^{*}_{a})(m_{\overline{a},a}(j_{a}(g)\otimes f ))$$

\noindent where $\text{coev}_{a}:\mathbbm{1}\rightarrow a\otimes \overline{a}$ and $\mathbbm{1}\rightarrow \overline{a}\otimes a$ and $\text{ev}_{a}: \overline{a}\otimes a\rightarrow \mathbbm{1}$ are standard solutions to the duality equations.\footnote{we can choose a unitary dual functor representing the canonical unitary spherical structure, see \cite{MR4133163}}. Note that the first inner product (which we call the left inner product) is conjugate linear in the second variable and linear in the first, while the second (which we call the right inner product) is conjugate linear in the first variable and linear in the second. Both are positive definite.

Then there is a unique positive operator $\Delta_{a}:\textbf{M}(a)\rightarrow \textbf{M}(a)$ (positive, in fact, with respect to both the left and right inner products) such that

$$_{a}\langle f, g\rangle=\langle g | \Delta_{a}(f)\rangle_{a}$$

The collection $\{\Delta_{a}:\mathbf{M}(a)\rightarrow \mathbf{M}(a)\}$ assemble into a natural endomorphsm on the functor $\mathcal{M}$, which we call the \textit{modular operator} on $\mathcal{M}$. Note $\Delta_{a}$ is diagonalizable for each $a \in \mathcal{C}$ and we  define

$$T(\textbf{M}):=\{\lambda\in \mathbbm{R}^{\times}_{+}\ :\ \lambda\ \text{is an eigenvalue for}\ \Delta_{a}\ \text{for some}\ a\in \mathcal{C}\}.$$

A connected W*-algebra object $\textbf{M}$ is \textit{tracial} if $T(\textbf{M})=\{1\}$, or equivalently, if for all $a\in \mathcal{C}$, $_{a}\langle f, g\rangle=\langle g | f\rangle_{a}$ (this can be taken as the definition of tracial for our purposes). If $(\mathcal{C}, \textbf{M})=\text{StdInv}(N\subseteq M, E)$ then $\textbf{M}$ is tracial if and only if $M$ is of type $\rm{II}_{1}$. In general, for any connected W*-algebra object, there is a canonical tracial subalgebra object $\textbf{M}_{tr}$ defined as a subfunctor by

$$\textbf{M}_{tr}(a):=\{v\in \textbf{M}(a)\ :\ \Delta_{a}(v)=v\}$$

\medskip

In the case $(\mathcal{C}, \textbf{M})=\text{StdInv}(N\subseteq M, E)$, let $M^{tr}$ denote the centralizer in $M$ of the state $\phi=\tau\circ E$. Then $M_{tr}$ is a $\rm{II}_{1}$ factor, and $N\subseteq M_{tr}$ is an irreducible, discrete unimodular inclusion (with the trace preserving conditional expectation). Let $\mathcal{C}_{tr}$ be the sub-tensor category of $\mathcal{C}$ tensor generated by irreducible summands $L^{2}(M_{tr})$. 

Since the modular operator of $M$ with respect to $\phi$ is the realization of the modular operator $\Delta$ associated to $\textbf{M}$ and $\phi$ is quasi-periodic, the centralizer of $\phi$ in $M$ is exactly the fixed points of of the modular operator in $M$. Thus we have 

$$(\mathcal{C}_{tr}, \textbf{M}_{tr})=\text{StdInv}(N\subseteq M_{tr}, E).$$

\begin{defn}\label{defineinvariant}
    Let $\textbf{M}$ be a tracial connected W*-algebra object in $\mathcal{C}$. Define $T_{0}(\textbf{M})$ to be the subset of $\mathbbm{R}^{\times}_{+}$ consisting of elements $\lambda\in T(\textbf{P})$, where $\mathbf{P}$ ranges over (isomorphism classes of) connected W*-algebra objects in $\text{Vec}(\mathcal{C})$ such that $\textbf{P}_{tr}\cong \textbf{M}$.
\end{defn}

Although $T_{0}(\textbf{M})$ is defined as a subset of $\mathbbm{R}^{\times}_{+}$, for the purposes of obtaining a realization obstruction, we can replace this with the subgroup generated by it. In all cases we analyze below, $T_{0}(\textbf{M})$ will already be a subgroup. This will be the invariant we use to obstruct realization of discrete subfactors.

Recall that if $M$ is a $\rm{II}_{1}$ factor, its fundamental group

$$\mathcal{F}(M):=\{t\in \mathbbm{R}^{\times}_{+}\ :\ M\cong M^{t}\}$$

\begin{prop}
\label{t0(m) included in f(m)}Let $N\subseteq M$ be an irreducible, discrete, unimodular inclusion of $\rm{II}_{1}$ factors with $\text{StdInv}(N\subseteq M)\cong (\mathcal{C},\textbf{M})$. Then $T_{0}(\textbf{M})\subseteq \mathcal{F}(M)$.
\end{prop}

\begin{proof} For any $\lambda\in T_{0}(\textbf{M})$, let $\textbf{P}$ be a corresponding connected W*-algebra  object with $\lambda\in T(\textbf{P})$. Then applying the realization functor and using that $\textbf{P}_{tr}\cong \textbf{M}$, we have $N\subseteq M\subseteq |\textbf{P}|$. But $M\subseteq |\textbf{P}|$ is a discrete decomposition [Ref realization of alg ob paper] and thus the eigenvalues of $\Delta_{\phi}$ on $M$ are elements the fundamental group of $M$.  By \cite{MR3948170}, the $\Delta_{\phi}$ is the realization of the categorical $\Delta$, and thus $\lambda\in \mathcal{F}(M)$. 

\end{proof}

\begin{cor}
    If $(\mathcal{C},\textbf{M})$ is an abstract, irreducible standard invariant, $M$ is an $\rm{II}_{1}$ factor, and $T_{0}(\textbf{M})\not\subset \mathcal{F}(M)$, then there is no irreducible, unimodular subfactor $N\subseteq M$ with $\text{StdInv}(N\subseteq M)\cong (\mathcal{C}, \textbf{M})$.
\end{cor}

\begin{rem}
Given $N\subseteq M$, it is natural to ask precisely which subset of $\mathcal{F}(M)$ is $T_{0}(M)$, in terms of the subfactor? Given $\lambda\in \mathcal{F}(M)$, one can construct a type $\rm{III}$ factor $P$ and a state whose centralizer is precisely $M$, and whose modular operator has eigenvalues $\{\lambda^{n}: n\in \mathbbm{Z}\}$. We will call $\lambda$ \textit{admissible} if $P$ can be chosen such that the inclusion $N\subseteq P$ is also an irreducible, discrete inclusion such that $\langle L^{2}(P,\phi)\rangle\subseteq \langle L^{2}(M,\tau)\rangle$. Then $T_{0}(\text{StdInv}(N\subseteq M))$ is precisely the set of admissible $\lambda$. 
\end{rem}

\begin{rem}\label{fundgroupex}
Another way to identify $T_{0}(\text{StdInv}(N\subseteq M))\subseteq \mathcal{F}(M)$ is to consider all invertible $M$-$M$ bimodules that restrict as $N$-$N$ bimodules to direct sums of bimodules in the unitary tensor category $\mathcal{C}:=\langle _{N}L^{2}(M)_{N}\rangle $. Then take the left Murray-von Neumann dimensions, which is a subgroup of the fundamental group. It is easy to see (from standard constructions dating back to Connes realizing fundamental group elements as modular eigenvalues of type III extensions \cite{MR0341115}) that this recovers our subgroup of the fundamental group. This suggests a more natural formulation of our invariant directly in terms of invertible Hilbert bimodules of $\mathcal{M}$ internal to $\mathcal{C}$ that parallels the ordinary definition of fundamental group, but a theory of internal tensor products for W*-algebra objects has not yet appeared in the literature. Thus we found it expedient to use the modular eigenvalue formulation in this paper.
\end{rem}

\section{Computing centralizers for $TLJ(\delta)$ subfactors.}
\label{computing t0 for tlj}
For $0<q\le 1$, we consider the unitary tensor category $\text{TLJ}(q+q^{-1})$. We typically refer to the parameter $q+q^{-1}=\delta$. This category can be realized as the representation category of the compact quantum group $\text{SU}_{-q}(2)$. Connected W*-algebra objects in this category were classified by De Commer and Yamashita (\cite[Theorem 2.4]{MR3420332}). They correspond to \textit{fair and balanced $\delta$-graphs}.

\begin{defn}
Let $\Gamma$ be a locally finite directed graph (where multiple edges are allowed) equipped with a function $w:E(\Gamma)\rightarrow \mathbbm{R}^{\times}_{+}$. The pair $(\Gamma,w)$ is a fair and balanced $\delta$-graph if

\begin{itemize}
\item 
There exists an involution $E(\Gamma)\rightarrow E(\Gamma)$, $e\mapsto \overline{e}$,  where $\text{source}(e) = \text{ target}(\overline{e})$ and $\text{source}(\overline{e}) = \text{ target}(e)$ and  $w(e)w(\overline{e})=1$.
\item 
For every vertex $x\in V(\Gamma)$, $\sum_{s(e)=x}w(e)=\delta$
\end{itemize}

\end{defn}

The main result of \cite{MR3420332} is that semisimple module categories for $\text{TLJ}(\delta)$ are classified by fair and balanced $\delta$-graphs, and connected W*-algebra objects are parameterized by triples $(\Gamma,w,*)$ where $(\Gamma,w)$ is a fair and balanced $\delta$-graph and $*$ is a choice of distingushed vertex (of course taken up to the obvious equivalences on both sides).

In the following section, note that by a path in $\Gamma$ we mean a sequence of edges $e_1, ..., e_k$ such that $s(e_{i}) = t(e_{i-1})$, where the $e_i$'s need not be distinct. If $s(e_1) = t(e_k)$, we call the path a loop. A loop based at a vertex $v$ is a loop with $s(e_1) = v$. The (total) weight of a path/loop is $w(e_1)...w(e_k)$.   

For the sake of completeness, we explain here how to explicitly realize a connected W*-algebra object in $\text{TLJ}(\delta)$ from $(\Gamma,w,*)$. Recall we can present $\text{TLJ}(\delta)$ as the unitary Karoubi completion of the category whose objects are natural numbers, and morphisms from $m$ to $n$ are linear combinations of Temperley-Lieb diagrams, modulo isotopy and the deletion that closed loops count for a factor of $\delta$.

Thus by naturality, to define a W*-algebra object it suffices to define it as a functor on this ``pre-completed" version of $\text{TLJ}(\delta)$. Let $L_{n}$ be the set of (oriented) loops, based of the vertex $*$, of length $n$. Then

$$\textbf{G}(n):=\mathbbm{C}[L_{n}]$$

Now, we must define the actions of the morphisms (i.e. the cups and caps) in $TLJ(\delta)$ under the functor $\mathbb{G}$. To do this, we first note that for an object $n \in TLJ(\delta)$, we have cups and caps for each strand postion which we can depict graphically as follows:
$$
\begin{tikzpicture}[scale=1, baseline = -.3cm]
\draw(0:0.5) arc (0:180:0.5);
\node(none) at (0.6, -0.1){$i$};
\node(none) at ( 1, 0) {$=$};
\end{tikzpicture}
\begin{tikzpicture}[scale=1, baseline = -.1cm]
\draw (-3,0) to (-3,1);
\draw (-2.5,0) to (-2.5,1);
\draw (-2,0) to (-2,1);
\node(none) at (-1.5, 0.5) {$\cdots$};
\draw (-1,0) to (-1,1);
\draw(0:0.5) arc (0:180:0.5);
\draw (1,0) to (1,1);
\node(none) at (1.5, 0.5) {$\cdots$};
\draw (2,0) to (2,1);
\node(none) at (-3, -0.5) {1};
\node(none) at (-2.5, -0.5) {2};
\node(none) at (-2, -0.5) {3};
\node(none) at (-1, -0.5) {$i-1$};
\node(none) at (-0.4, -0.4) {$i$};
\node(none) at (1, -0.5) {$i+2$};
\node(none) at (2, -0.5) {$n$};
\node(none) at (-3, 1.5) {1};
\node(none) at (-2.5, 1.5) {2};
\node(none) at (-2, 1.5) {3};
\node(none) at (-1, 1.5) {$i-1$};
\node(none) at (1, 1.5) {$i$};
\node(none) at (2, 1.5) {$n-2$};
\end{tikzpicture}
\hspace{2em}
\begin{tikzpicture}[scale=1, baseline = -0.7cm]
\draw(180:0.5) arc (-180:0:0.5);
\node(none) at (0.6, -0.5){$i$};
\node(none) at ( 1, -0.2) {$=$};
\end{tikzpicture}
\begin{tikzpicture}[scale=1, baseline = -1.1cm]
\draw (-3,-1) to (-3,0);
\draw (-2.5,-1) to (-2.5,0);
\draw (-2,0) to (-2,-1);
\node(none) at (-1.5, -0.5) {$\cdots$};
\draw (-1,0) to (-1,-1);
\draw(180:0.5) arc (-180:0:0.5);
\draw (1,0) to (1,-1);
\node(none) at (1.5, -0.5) {$\cdots$};
\draw (2,0) to (2,-1);
\node(none) at (-3, -1.5) {1};
\node(none) at (-2.5, 0.5) {2};
\node(none) at (-2, 0.5) {3};
\node(none) at (-1, 0.5) {$i$};

\node(none) at (1, 0.5) {$i+3$};
\node(none) at (2, 0.5) {$n + 2$};
\node(none) at (-3, 0.5) {1};
\node(none) at (-2.5, -1.5) {2};
\node(none) at (-2, -1.5) {3};
\node(none) at (-1, -1.5) {$i$};
\node(none) at (1, -1.5) {$i + 1$};
\node(none) at (2, -1.5) {$n $};
\end{tikzpicture}
$$
From the pictures above, we see that $\mathbf{G}(\cup_i): \mathbb{C}[L_n] \to \mathbb{C}[L_{n+2}]$ and $\mathbf{G}(\cap_i): \mathbb{C}[L_n] \to \mathbb{C}[L_{n-2}]$. We will define these maps on a basis element $\ell = e_1e_2...e_n$.
\begin{equation}
\label{cup equation for g}
   \mathbf{G}(\cup_i) = \sum_{e \in E(\Gamma), s(e) = t(e_i)} w(e)^{\frac{1}{2}}(e_1, ..., e_i, e, \overline{e}, e_{i+1}, ..., e_n) 
\end{equation}
\begin{equation}
\label{cap equation for g}
    \mathbf{G}(\cap_i) = \delta_{e_i = \overline{e}_{i+1}} w(e_i)^{\frac{1}{2}}(e_1, ..., e_{i-1}, e_{i + 2}, ..., e_n)
\end{equation}
The coefficents of $w(e)^{\frac{1}{2}}$ ensure that these maps satisfy the relations of $TLJ(\delta)$. If $l=e_1 e_{2}\dots e_{n}\in L_{n}$ define $w(l)=\prod_{e_i\in l} w(e_{i})$. Also define $\overline{l}=\overline{e}_{n}\dots \overline{e}_{1}$. The $*$-structure on $\textbf{G}$ is defined by the conjugate linear extension of the map given on the loop basis by $j_{n}(l)=w(\overline{l})^{\frac{1}{2}}\overline{l}$.

Then by \cite[Proof of Proposition 6.8]{MR3948170} the eigenvalues of the $\Delta$-operator are the subgroup

$$T(\textbf{G}):=\{w(l)\ :\ l\in L_{n}\ \text{for some $n$}\}$$

\begin{defn}
    A fair and balanced $\delta$-graph $\Gamma$ is tracial if $T(\Gamma) = \{1\}$, that is, if $w(l) = 1$ for all loops $l$ in $L_n$ for every $n$. 
\end{defn}

\begin{ex}
\label{double chain graph intro} 
    Consider the following graph $\Gamma_1$, where $a + a^{-1} + b + b^{-1} = \delta$ (we can choose such $a$ and $b$ for any $\delta > 4$):
    $$
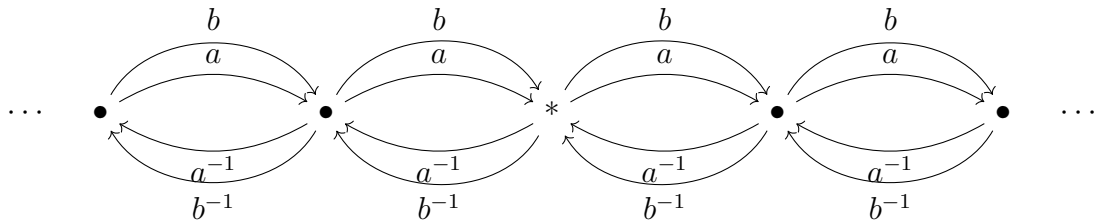
\begin{tikzpicture}
        \node(none) (s) at (0,0){*} ;
        \node(none) (v)at (3,0){$\bullet$};
        \node(none) (w) at (-3,0){$\bullet$};
        \node(none) (x)at (6,0){$\bullet$};
        \node(none) (y) at (-6,0){$\bullet$};
        \node(none) (z) at (-7,0){$\cdots$};
        \node(none) (z) at (7,0){$\cdots$};
        \node(none) (1) at (-1.5,0.75){$a $};
        \node(none) (2) at (-1.5,1.25){$b $};
        \node(none) (3) at (1.5,0.75){$a $};
        \node(none) (4) at (1.5,1.25){$b $};

         \node(none) (1) at (-4.5,0.75){$a $};
        \node(none) (2) at (-4.5,1.25){$b $};
        \node(none) (3) at (4.5,0.75){$a $};
        \node(none) (4) at (4.5,1.25){$b $};
        
        \node(none) (1) at (-1.5,-0.75){$a^{-1}$};
        \node(none) (2) at (-1.5,-1.25){$b^{-1}$};
        \node(none) (3) at (1.5,-0.75){$a^{-1} $};
        \node(none) (4) at (1.5,-1.25){$b^{-1} $};
         
         \node(none) (1) at (-4.5,-0.75){$a^{-1}$};
        \node(none) (2) at (-4.5,-1.25){$b^{-1}$};
        \node(none) (3) at (4.5,-0.75){$a^{-1} $};
        \node(none) (4) at (4.5,-1.25){$b^{-1} $};
        
         \draw [->] (w) to [out=30,in=150] (s) ;
        \draw [->] (w) to [out=60,in=120] (s) ;
        \draw [->] (s) to [out=210,in=330] (w) ;
        \draw [->] (s) to [out=240,in=300] (w) ;
        
        \draw [->] (s) to [out=30,in=150] (v) ;
        \draw [->] (s) to [out=60,in=120] (v) ;
        \draw [->] (v) to [out=210,in=330] (s) ;
        \draw [->] (v) to [out=240,in=300] (s) ;

        \draw [->] (v) to [out=30,in=150] (x) ;
        \draw [->] (v) to [out=60,in=120] (x) ;
        \draw [->] (x) to [out=210,in=330] (v) ;
        \draw [->] (x) to [out=240,in=300] (v) ;

         \draw [->] (y) to [out=30,in=150] (w) ;
        \draw [->] (y) to [out=60,in=120] (w) ;
        \draw [->] (w) to [out=210,in=330] (y) ;
        \draw [->] (w) to [out=240,in=300] (y) ;
    \end{tikzpicture}$$

Note that in $\Gamma_1$, any loop $\ell$ based at * has some furthest point $v$ from *, and the vertices in the path from * to $v$ are exactly those in the return path from $v$ to *, appearing in reverse order.  We can commute factors in $w(\ell)$ to pair the weight of the edge $x \to y$ with the weight of the edge $y \to x$ for every pair of vertices $x,y$ in the path from * to $v$. The product of these weight is either $aa^{-1}$, $bb^{-1}$, $ab^{-1}$ or $ba^{-1}$. The first two expressions are 1, and the last two are inverses of each other, so $w(\ell) = (ab^{-1})^{k}$ for some integer $k$. So we can say $T(\Gamma_1) = \langle ab^{-1}\rangle$.
\end{ex}

\begin{ex}
    \label{single chain graph intro}
    Consider the following graph $\Gamma_2$, where $q + q^{-1} = \delta$ (we can find such $q$ for $\delta > 2$):
     $$
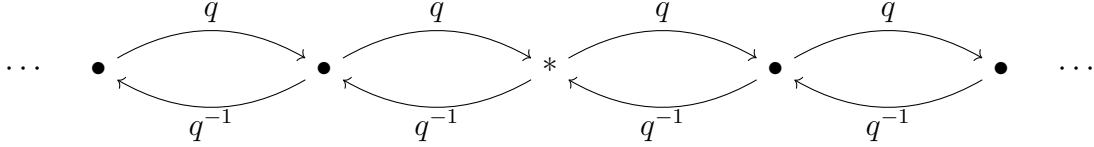
\begin{tikzpicture}
        \node(none) (s) at (0,0){*} ;
        \node(none) (v)at (3,0){$\bullet$};
        \node(none) (w) at (-3,0){$\bullet$};
        \node(none) (x)at (6,0){$\bullet$};
        \node(none) (y) at (-6,0){$\bullet$};
        \node(none) (z) at (-7,0){$\cdots$};
        
        \node(none) (z) at (7,0){$\cdots$};
        \node(none) (1) at (-1.5,0.75){$q$};
        \node(none) (3) at (1.5,0.75){$q $};
         \node(none) (1) at (-4.5,0.75){$q$};
       `\node(none) (3) at (4.5,0.75){$q$};
        \node(none) (1) at (-1.5,-0.75){$q^{-1}$};
        \node(none) (3) at (1.5,-0.75){$q^{-1} $};
         \node(none) (1) at (-4.5,-0.75){$q^{-1}$};
        \node(none) (3) at (4.5,-0.75){$q^{-1} $};

         \draw [->] (w) to [out=30,in=150] (s) ;
        \draw [->] (s) to [out=210,in=330] (w) ;

        \draw [->] (s) to [out=30,in=150] (v) ;
        \draw [->] (v) to [out=210,in=330] (s) ;

        \draw [->] (v) to [out=30,in=150] (x) ;
        \draw [->] (x) to [out=210,in=330] (v) ;

         \draw [->] (y) to [out=30,in=150] (w) ;
        \draw [->] (w) to [out=210,in=330] (y) ;
    \end{tikzpicture}$$

See that $T(\Gamma_2) = \{1\}$, since each time we add an edge to a loop, the only way to get back to $*$ is to add the inverse edge (with inverse weight) to the loop as well. This means that the weight of the loop will be 1. 
\end{ex}
 \begin{remark}
     Note that if $\Gamma$ is a tracial fair and balanced $\delta$-graph, then its weighting $\omega$ on edges induces a well-defined weighting $w_V$ on the vertices of $\Gamma$ by labeling any vertex $v$ with the weight of any path from $*$ to $v$ ($\Gamma$ being tracial must necessarily mean all paths to $v$ have the same weight). This weighting has the property that if $e$ is the edge with souce $v_1$ and target $v_2$, then $w(e) = \frac{w_V(v_2)}{w_V(v_1)}$. The existence of a vertex weighting with this property is equivalent to a fair and balanced $\delta$-graph being tracial \cite{MR3948170}.
 \end{remark}

\subsection{Tracial covers of fair and balanced $\delta$-graphs}

\label{tracial cover subsection}

Associated to a fair and balanced $\delta$-graph $\Gamma$ with related $W^*$-algebra object $\textbf{M}$ is a \textit{tracial} W*-algebra object $\textbf{M}_{tr}$, which consists of the eigenvectors of the modular operator $\Delta$ with eigenvalue $1$. By the classification, this has an associated tracial fair and balanced $\delta$ graph representation, which we call $\Gamma_{tr}$. Our goal is to identify this graph from the data of the original graph. As we will see, $\Gamma_{tr}$ can be thought of as a ``covering space" for the graph $\Gamma$. We will describe an algorithmic presentation of $\Gamma_{tr}$.

\subsubsection{Constructing $\Gamma_{tr}$}\label{construction tracial cover}

For the rest of this section, we will assume that $\Gamma$ is a fair and balanced $\delta$-graph with a distinguished vertex $* \in V(\Gamma)$ and edge weighting $\omega : E(\Gamma) \rightarrow \mathbb{R}^{\times}_{>0}$. 

First, define the graph $\Gamma'$, where the vertices in $\Gamma'$ correspond to paths with initial point $*$ in $\Gamma$. Edges $\tilde{e}:p_{1}\rightarrow p_{2}$ correspond to edges $e: t(p_{1})\rightarrow t(p_{2})$ such that $p_{2}=p_{1}*e$. $w(\widetilde{e})=w(e)$. 

Note that $\Gamma^{\prime}$ is fair and balanced. Indeed, for the vertex $p_1$, we will have exactly one edge in $\Gamma'$ corresponding to each edge in $\Gamma$ whose source $t(p_1)$.

We can define an equivalence relation $\sim_{tr}$ on the vertices of $\Gamma'$ where $v_1 \sim v_2$ if the corresponding paths have the same target and  weight. 

\begin{defn}
    \label{tracial cover def}
    $\Gamma_{tr}$ is the graph where vertices are eqiuvalence classes of $\Gamma'$ under $\sim_{tr}$. We can index the vertices $\Gamma_{tr}$ corresponding to the equivalence class of paths with weight $\lambda$ and target in $\Gamma$ of $v$ as $[\lambda,v]$. For a vertex $[\lambda,v]$ in $V(\Gamma_{tr})$, the edges with source $[\lambda,v]$ and target $[\lambda^{\prime},v^{\prime}]$ correspond to edges $e\in E(\Gamma)$ with $s(e)=v$ and $t(e)=v^{\prime}$, and $w(e) =\frac{\lambda^{\prime}}{\lambda}$. The weight of the edge is simply its weight in $\Gamma$. We call $\Gamma_{tr}$ the tracial cover of the graph $\Gamma$. 
\end{defn}

Note that $\Gamma_{tr}$ is indeed well-defined. If we have some edge $[\lambda_{1}, v_{1}] \to [\lambda_2,{v_2}]$ in $\Gamma_{tr}$, by definition there is an edge $e: v_1 \to v_2$ with weight $\epsilon=\frac{\lambda_{2}}{\lambda_{1}}$ in $\Gamma$. Since every element in the class $[\lambda_1,{v_1}]$ ends at $v_1$, we can compose $e$ with any representative $p$ and get a path with target $v_2$ and weight $\lambda_2$. So for any representative of $[\lambda_1, v_1]$ we have an edge in $\Gamma'$ of weight $\epsilon$ to some representative of $[\lambda_2,{v_2}]$ and thus edges are well-defined.

\begin{prop}
$\Gamma_{tr}$ is a tracial fair and balanced $\delta$-graph. 
\end{prop}

\begin{proof}

To show $\Gamma_{tr}$ is a fair and balanced $\delta$-graph, we first show that we have the required involution $E(\Gamma_{tr}) \to E(\Gamma_{tr})$. To do this, note any edge $e_{tr} :[\lambda_1,{v_1}] \to [\lambda_2,{v_2}]$ corresponds to an edge $e: v_1 \to v_2$ in $\Gamma$ with weight $w(e)$. Since $\Gamma$ is a fair and balanced $\delta$ graph, there exists an edge $\overline{e}: v_2 \to v_1$ with weight $w(e)^{-1}$. Now, this edge will induce an edge $[\lambda_2,{v_2}] \to [\lambda_1,{v_1}]$ of weight $w(e)^{-1}$ in $\Gamma_{tr}$, which we can choose as $\overline{e_{tr}}$, since $w(e_{tr}) = w(e)$ and so $w(e_{tr})w(\overline{e_{tr}}) = 1$ as desired. 

The edges with source $[\lambda,{v}]$ correspond   to edges with source $v$ in $\Gamma$, so the sum of their weights are the same. Since $\Gamma$ is a fair and balanced $\delta$-graph, so is $\Gamma_{tr}$.

To see that $\Gamma_{tr}$ is tracial, we show that the weight function admits a vertex weighting representation. Define $\nu([\lambda, v])=\lambda$. Then by definition, for any edge $e:[\lambda_{1},v_{1}]\rightarrow [\lambda_{2}, v_{2}]$, $w(e)=\frac{\lambda_{2}}{\lambda_{1}}=\frac{\nu(t(e))}{\nu(s(e))}$.

\end{proof}


\begin{thm}
Suppose $\mathbf{M}$ is a W*-algebra object in $TLJ(\delta)$ with fair and balanced $\delta$-graph $\Gamma$. Then $\textbf{M}_{tr}$ has fair and balanced $\delta$-graph $\Gamma_{tr}$
\end{thm}

\begin{proof}
Our strategy for showing this isomorphism is to first define a map between loops in $\Gamma$ of weight 1 (a basis for $\textbf{M}_{tr}$) and loops in $\Gamma_{tr}$ (a basis for the algebra object represented by $\Gamma_{tr}$) and show that this map is a bijection. We will then show that this map extends to an algebra ismorphism by showing that respects multiplcation (concatenation of loops) and the previous defined action of cups and caps. 

There is a map $\mu$ between weight 1 loops of length $n$ in $\Gamma$ to loops of length $n$ in $\Gamma_{tr}$ because every vertex $v^{\prime}$ in $\Gamma$ has exactly the same outgoing edges as any $[\lambda, v^\prime]$in $\Gamma_{tr}$. So $\mu$ will trace the loop from $\Gamma$ in $\Gamma_{tr}$. 

    To show that $\mu$ is injective, suppose we have two distinct loops $\ell_1$ and $\ell_2$ of weight 1 in $\Gamma$. At some point, these loops must differ by an edge. Say this occurs at a vertex $v$, where the partial weight of the loop is $\epsilon$. Say that the next edge of $\ell_1$  has weight $\delta_1$ and the next edge of $\ell_2$ has weight $\delta_2$. In $\Gamma_{tr}$, $\mu(\ell_1)$ will have an edge to $ [\epsilon, v]\to [\epsilon\delta_1, v_1]$  and $\mu(\ell_2)$ will have an edge $ [\epsilon, v]\to [\epsilon\delta_2, v_2]$. These are different vertices in $\Gamma_{tr}$, even if $v_1 = v_2$. So $\mu(\ell_1) \neq \mu (\ell_2)$ and $\mu$ is injective.
 
    To show that $\mu$ is surjective, consider a loop $\ell'$ in $\Gamma_{tr}$. We must construct a loop in $\Gamma$ that maps to $\ell'$. This is simple, as when we travel along an edge $[\lambda_1,{v_1}] \to [\lambda_2,{v_2}]$ in $\Gamma_{tr}$, we can just add the edge of weight $\frac{w_2}{w_1}$ from $v_1$ to $v_2$ into our loop in $\Gamma$. So $\mu$ is surjective and therefore a bijection. 

    It is fairly easy to see that $\mu$ respects multiplication. If $\ell_1 = e_1, e_2, ..., e_k$ (with vertices $v_1$,...,$v_{k-1}$) and $\ell_2 = f_1, ..., f_j$ (with vertices $w_1$,...,$w_{j-1}$), then $\ell_1 * \ell_2 = e_1...e_kf_1...f_j$. Now, $\mu(\ell_1*\ell_2)$ is the loop
    $$* \to [\omega(e_1), v_1] \to [ \omega(e_1)\omega(e_2), v_2] \to... \to [ \omega(e_1)...\omega(e_{k-1}), v_{k-1}] \to * \to $$
    $$ [w_1, \omega(f_1)] \to [w_2, \omega(f_1)\omega(f_2)] \to .. \to [w_{k-1}, \omega(f_1)...\omega(f_{j-1})] \to * $$

    We see that the first line of this loop is exactly $\mu(\ell_1)$ and the second is $\mu(\ell_2)$, so this loop is both $\mu(\ell_1 * \ell_2)$ and $\mu(\ell_1) * \mu(\ell_2)$ and $\mu$ respects multiplication in the algebra. 

    What remains is to show that the map $\mu$ is compatibile with cups and caps (and specfically the maps $\mathbf{G}(\cap_i)$ and $\mathbf{G}(\cup_i)$.   We start by showing compatibilty with $\mathbf{G}(\cup_i)$. Suppose $\ell  = * \to v_1 \to v_2 \to ... \to v_{k-1} \to v_k \to ... \to *$, where $e_i$ has source $v_{i-1}$ and target $v_i$. The loop $\mathbf{G}(\cup_{k-1})(\ell)$ in $G$ is  $\delta_{\overline{e_k} = e_{k+1}} (e_1, ..., e_{k-1}, e_{k+2})$. Since if $\overline{e_k} \neq e_{k+1}$ this loop is 0 and gets mapped to the trivial loop in $\Gamma_{tr}$, assume $\overline{e_k} = e_{k+1}$. This means that $v_{k-1} = v_{k+1}$ and $w(e_k) = w(e_{k+1})^{-1}$ and we can delete these edges and scale $\mathbf{G}(\cup_{k-1})(\ell)$ by $w(e_k)^{\frac{1}{2}}$. Now, $\mu(\mathbf{G}(\cup_{k-1})(\ell))$ is the loop $* \to [w(e_1),{v_1}] \to ... \to [w(e_1)w(e_2)...w(e_{k-1}),{v_{k-1}}] \to  [w(e_1)...w(e_{k-1})w(e_{k+2}),{v_{k+2}}] \to ... \to *$ with coefficent $w(e_k)^{\frac{1}{2}}$.

   Note $\mu(\ell) = * \to [w(e_1),{v_1}]\to [w(e_1)w(e_2),{v_2}] \to ... \to [w(e_1)...w(e_{k-1}),{v_{k-1}}] \to [w(e_1)...w(e_{k}),{v_k}] \to [ w(e_1)...w(e_{k+1}),{v_{k + 1}}] \to * $, with edge weights indentical to the $e_i$'s. Now, we must apply $\mathbf{G}_v(\cup_{k-1})(\ell)$. We will argue below that $\overline{e_{k}} = e_{k+1}$ in $\Gamma_{tr}$ exactly when $\overline{e_k} = e_{k+1}$ in $\Gamma$. So $\mathbf{G}(\cup_{k-1})(\mu(\ell)) = 0$ if and only if $\mu(\mathbf{G}(\cup_{k-1})(\ell)) = 0$. If $\overline{e_k} = e_{k+1}$, then we can delete these edges and scale by  $w(e_k)^{\frac{1}{2}}$ as before. We get the loop $* \to [w(e_1),{v_1}] \to [w(e_1)w(e_2),{v_2}] \to ... \to [w(e_1)...w(e_{k-1}),{v_{k-1}}] \to [w(e_1)... w(e_{k})w(e_{k+1})w(e_{k+2}),{v_{k+2}}] \to * $. But since $w(e_{k})^{-1} = w(e_{k+1})$, $ [w(e_1)... w(e_{k})w(e_{k+1})w(e_{k+2}),{v_{k + 2}}] $ \newline $= [w(e_1)...w(e_{k-1})w(e_{k+2}),{v_{k + 2}}]$ and this is exactly the loop we got in the previous paragraph. So $\mu$ is compatible with $\mathbf{G}(\cup_{k-1})$.
    \\
    
    We will now justify the claim that $\overline{\mu(e_{k})} = \mu(e_{k+1})$ in $\Gamma_{tr}$ exactly when $\overline{e_k} = e_{k+1}$ in $\Gamma$. In $\Gamma_{tr}$, $\mu(e_k)$ has source $[w(e_1)...w(e_{k-1}), v_{k-1}]$ and target $[w(e_1)...w(e_k), v_{k}]$. Likewise, $\mu(e_{k+1})$ has source $[w(e_1)...w(e_{k}), v_{k}]$ and target $[w(e_1)...w(e_{k+1}, v_{k+1}]$. If $\overline{e_k} = e_{k+1}$, then $v_{k+1} = v_{k-1}$ and $w(e_k) =  w(e_{k+1})^{-1}$. So $[w(e_1)...w(e_{k-1}), v_{k-1}] = [w(e_1)...w(e_{k+1}), v_{k+1}] $ and so $\mu(e_k) = \overline{\mu(e_{k+1})}$. Now if $\mu(e_k) = \overline{\mu(e_{k+1})}$, this implies that $[w(e_1)...w(e_{k-1}), v_{k-1}] = [w(e_1)...w(e_{k+1}), v_{k+1}] $. These vertices are only equal if $v_{k+1} = v_{k-1}$ and $w(e_1)...w(e_{k-1}) = w(e_1)...w(e_{k+1})$ (which further implies $w(e_k)w(e_{k+1})$). Thus, $e_k = \overline{e_{k+1}}$. This justifies our previous claim.
    
    To show that our bijection is compatible with the cup operation, we continue to consider the loop $\ell = * \to v_1 \to v_2 \to ... \to v_{k-1} \to v_k \to ... \to *$, where $e_i$ has source $v_{i-1}$ and target $v_i$. So $$\mathbf{G}(\cap_k)(\ell)= \sum_{s(e) = v_k} w(e)^{\frac{1}{2} }(* \to v_1 \to v_2 \to ... \to v_k \to t(e) \to v_k \to ...*)$$

    Now,  $\mu (\mathbf{G}(\cap_k)(\ell)$ is  $$\sum_{s(e) = [w_k, v_k]} w(e)^{\frac{1}{2} }(\ell_e)$$, where $w_k = w(e_1)...w(e_k)$  and  $$ \ell_e  = * \to  [w(e_1),{v_1}] \to ... \to [w(e_1)...w(e_k),{v_k}] \to [ w(e_1)...w(e_k)w(e),{t(e)}] \to $$ $$ [v_k, w(e_1)...w(e_k)w(e)w(\overline{e}) = w(e_1)...w(e_k),{v_k}] \to [w(e_1)...w(e_{k+ 1}),{v_{k+1}}] \to ... \to *, $$ . 
    \\

    The image of $\ell$ under $\mu$ is  $* \to [w(e_1),{v_1}] \to  [w(e_1)w(e_2),{v_2}] \to ... \to [w(e_1)...w(e_{k-1}),{v_{k-1}}] \to [w(e_1)...w(e_{k}),{v_k}] \to [w(e_1)...w(e_{k+1}),{v_{k+1}}] \to * $, with edge weights correspoding to weights of $e_i$.  
    We now must consider the sum 
    $$  \sum_{s(e) = [w(e_1)...w(e_k), v_k]} w(e)^{\frac{1}{2} }\ell_e$$ where an edge $e$ has target $[w(e_1)...w(e_k)w(e), t]$ and $$ \ell_e  = * \to  [w(e_1),{v_1}] \to ... \to [w(e_1)...w(e_k),{v_k}] \to [ w(e_1)...w(e_k)w(e),t] \to $$ $$ [ w(e_1)...w(e_k)w(e)w(e)^{-1} = w(e_1)...w(e_k),{v_k}] \to [w(e_1)...w(e_{k+ 1}),{v_{k+1}}] \to ... \to *, $$
    Notice that these loops are the same as the ones previous constructed. Furthermore, we have exactly one edge with source $[w(e_1)...w(e_k), v_k]$ for every edge with source $v_k$ in $\Gamma$, and these edges have indentical weights. So the two summations we have constructed are equal and $\mu$ is compatible with the cap operation.   

    So our map $\mu$ is compatible with all algebra operations is therefore an isomorphism of algebra objects as desired. 
    
\end{proof}

Notice that the weight of loops in $\Gamma$ correspond to eigenvalues of the $\Delta$-operator, this theorem tells us that the relationship between $\Gamma$ and $\Gamma_{tr}$ is equivalent to that between $\mathbf{M}$ and $\mathbf{M}^{tr}$. 

\begin{ex}
  We consider this construction for the graph $\Gamma_1$ from example \ref{double chain graph intro}. In the first step, we will construct the following graph  $\Gamma'$ (here, and onward through this example, we assume that every drawn edge has an opposite edge of inverse weight without explicity drawing the edge):

    $$
    
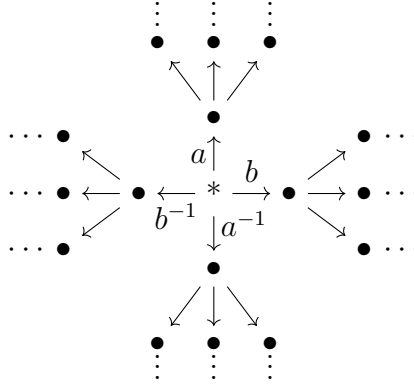
\begin{tikzpicture}
     \node(none) (s) at (0,0){*} ;
     \node(none) (a) at (1, 0){$\bullet$};
     \node(none) (a-1) at (-1, 0){$\bullet$};
     \node(none) (b) at (0, 1){$\bullet$};
     \node(none) (b-1) at (0, -1){$\bullet$};

     \draw[->] (s) to (a);
     \draw[->] (s) to (a-1);
     \draw[->] (s) to (b);
     \draw[->] (s) to (b-1);

    \node(none) (aa) at (2,0) {$\bullet$};
    \node(none) (ab) at (2,0.75) {$\bullet$};
    \node(none) (ab-1) at (2,-0.75) {$\bullet$};
    \node(none) (a-1a-1) at (-2,0) {$\bullet$};
    \node(none) (a-1b) at (-2,0.75) {$\bullet$};
    \node(none) (a-1b-1) at (-2,-0.75){$\bullet$};
    \node(none) (bb) at (0,2) {$\bullet$};
    \node(none) (ba) at (0.75,2) {$\bullet$};
    \node(none) (ba-1) at (-0.75,2) {$\bullet$};

    \node(none) (b-1b-1) at (0,-2) {$\bullet$};
    \node(none) (b-1a) at (0.75,-2) {$\bullet$};
    \node(none) (b-1a-1) at (-0.75,-2) {$\bullet$};

     \draw[->] (a) to (aa);
     \draw[->] (a) to (ab);
     \draw[->] (a) to (ab-1);
     
     \draw[->] (a-1) to (a-1a-1);
     \draw[->] (a-1) to (a-1b);
     \draw[->] (a-1) to (a-1b-1);

    \draw[->] (b) to (bb);
     \draw[->] (b) to (ba);
     \draw[->] (b) to (ba-1);
     
     \draw[->] (b-1) to (b-1b-1);
     \draw[->] (b-1) to (b-1a);
     \draw[->] (b-1) to (b-1a-1);

    \node(none) (aa...) at (2.5,0) {$\cdots$};
    \node(none) (ab...) at (2.5,0.75) {$\cdots$};
    \node(none) (ab-1...) at (2.5,-0.75) {$\cdots$};
    \node(none) (a-1a-1...) at (-2.5,0) {$\cdots$};
    \node(none) (a-1b...) at (-2.5,0.75) {$\cdots$};
    \node(none) (a-1b-1...) at (-2.5,-0.75){$\cdots$};
    \node(none) (bb...) at (0,2.5) {$\vdots$};
    \node(none) (ba...) at (0.75,2.5) {$\vdots$};
    \node(none) (ba-1...) at (-0.75,2.5) {$\vdots$};

    \node(none) (b-1b-1...) at (0,-2.2) {$\vdots$};
    \node(none) (b-1a...) at (0.75,-2.2) {$\vdots$};
    \node(none) (b-1a-1...) at (-0.75,-2.2) {$\vdots$};

    \node(none) at (-0.2, 0.5) {$a$};
    \node(none) at (0.4, -0.4) {$a^{-1}$};
    \node(none) at (-0.5, -0.3) {$b^{-1}$};
    \node(none) at (0.5, 0.3) {$b$};
    \end{tikzpicture}
    $$
When we take the quotient of this graph described, we see that ${\Gamma_1}_{tr}$ is the following graph. 

$$
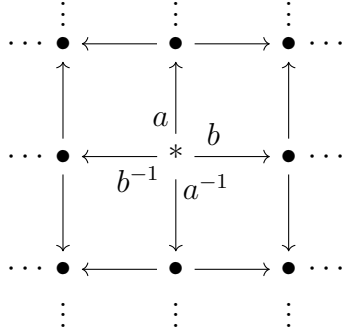
\begin{tikzpicture}
    \node(none) (s) at (0,0){*} ;
     \node(none) (a) at (1.5, 0){$\bullet$};
     \node(none) (a-1) at (-1.5, 0){$\bullet$};
     \node(none) (b) at (0, 1.5){$\bullet$};
     \node(none) (b-1) at (0, -1.5){$\bullet$};
     \node(none) (ab) at (1.5,1.5){$\bullet$};
     \node(none) (a-1b) at (-1.5,1.5){$\bullet$};
     \node(none) (ab-1) at (1.5,-1.5){$\bullet$};
     \node(none) (a-1b-1) at (-1.5,-1.5){$\bullet$};

      \draw[->] (s) to (a);
     \draw[->] (s) to (a-1);
     \draw[->] (s) to (b);
     \draw[->] (s) to (b-1);
      \draw[->] (a) to (ab);
     \draw[->] (a) to (ab-1);
     \draw[->] (a-1) to (a-1b-1);
    \draw[->](a-1) to (a-1b);
     \draw[->] (b) to (ab);
     \draw[->] (b-1) to (ab-1);
     \draw[->] (b) to (a-1b);
     \draw[->] (b-1) to (a-1b-1);

    \node(none) (aa...) at (2,0) {$\cdots$};
    \node(none) (ab...) at (2,1.5) {$\cdots$};
    \node(none) (ab-1...) at (2,-1.5) {$\cdots$};
    \node(none) (a-1a-1...) at (-2,0) {$\cdots$};
    \node(none) (a-1b...) at (-2,1.5) {$\cdots$};
    \node(none) (a-1b-1...) at (-2,-1.5){$\cdots$};
    \node(none) (bb...) at (0,2) {$\vdots$};
    \node(none) (ba...) at (1.5,2) {$\vdots$};
    \node(none) (ba-1...) at (-1.5,2) {$\vdots$};

    \node(none) (b-1b-1...) at (0,-2) {$\vdots$};
    \node(none) (b-1a...) at (1.5,-2) {$\vdots$};
    \node(none) (b-1a-1...) at (-1.5,-2) {$\vdots$};

    \node(none) at (-0.2, 0.5) {$a$};
    \node(none) at (0.4, -0.4) {$a^{-1}$};
    \node(none) at (-0.5, -0.3) {$b^{-1}$};
    \node(none) at (0.5, 0.3) {$b$};
\end{tikzpicture}$$
\end{ex}

\subsection{Actions of Groups on Tracial Graphs}
\label{actions of groups on graphs}
Now, consider a tracial fair and balanced $\delta$-graph $\Gamma$ and group $H \leq \mathbb{R}^{*}$. Recall that since $\Gamma$ is tracial, it has a vertex weighting $ w_V: V(\Gamma) \to \mathbb{R}^{\times}_{>0}$ where $w(e) = \frac{w_V(t(e))}{w_V(s(e))}$ Suppose $H$ has an action on the vertices of $\Gamma$ so that $w_V(h \cdot v) = h*w_V(v)$ for every $v \in V(\Gamma)$. We will construct a graph $\Gamma^{H}$ whose tracial cover is $\Gamma$.

\begin{defn}
\label{gamma_H}
    If $\Gamma$ and $H$ are as above, $\Gamma^{H}$ is the graph whose vertices are orbits of $V(G)$ under the action of $H$. For edges, take an orbit $[x]$ and pick a  representative $x$. There is one edge $[x] \to [y]$ for every edge $x \to y^\prime \in [y]$ in $\Gamma$, with the same weight as the edge in $\Gamma$. 
\end{defn}

The edges of $\Gamma^{H}$ are well-defined. If we have some other representative $x' = h' \cdot x \in [x]$, any edge $x' \to y'$ in $\Gamma$ for $y' \in [y]$ will correspond to an edge $x \to h^{'-1}y'$ in $\Gamma$, since the action of $H$ on $\Gamma$ preserves adjaceny. The one-to-one correspondence between  edges with source $x$ in $\Gamma$ and  edges with source $[x]$ in $\Gamma^{H}$,so $\Gamma^{H}$ is a fair and balanced $\delta$-graph.

However, $\Gamma^{H}$ is not a  tracial graph. If we have edges $e_1, e_2: [x] \to [y]$  in $\Gamma^{H}$, they correspond to two distinct edges $x \to y_1$, $x \to y_2$ in $\Gamma$, where $y_1 \neq y_2$. Therefore $y_2 = h \cdot y_1$ for some $h \neq 1 \in H$.  So in $\Gamma^{H}$, $w(e_1) \neq w(e_2)$ since $w(e_1) = \frac{w_V(y_1)}{w-V(x)}$ and $w(e_2) = \frac{w_V(y_2)}{w(x)} = \frac{h * w_V(y_1)}{w_V(x)}$. 
\\

Because we now have a non-tracial graph $\Gamma^{H}$. We can follow the procedure outlined in section \ref{construction tracial cover} to find it's tracial cover.

\begin{thm}
    The tracial cover of $\Gamma^{H}$ is $\Gamma$.
\end{thm}
\begin{proof}
First, we show that for $v \in V(\Gamma)$, that there is a well defined map $v \mapsto [w_V(v),{orb(v)}] \in V(\Gamma^{H}_{tr})$ . First, see that under the action of $H$, $v$ corresponds to a vertex $orb(v) \in \Gamma^{H}$. Now, $v$ having a weight of $w_V(v)$ in $\Gamma$ means there is a path $* \to ... \to v$ in $\Gamma$ with weight $w_V(v)$ and so there is a path $orb(*) \to ... \to orb(v)$ with weight $w_V(v)$ in $\Gamma^{H}$. This means $[w_V(v),{orb(v)}]$ is in $\Gamma_{tr}^{H}$.  

To show that this map is injective, consider vertices $v_1$ and $v_2$ in $\Gamma$.Either $orb(v_1) = orb(v_2)$, in which case $w(v_1) \neq w(v_2)$,  or $orb(v_1) \neq orb(v_2)$. Thus, $[w(v_1),{orb(v_1)}] \neq [w(v_2) ,{orb(v_2)}]$. To show that is surjective, see that a vertex in $\Gamma_{tr}^{H}$ is denoted by $[w,v]$ for $v \in V(\Gamma^{H})$ (where $v = orb(v^\prime)$ for $v^\prime \in \Gamma$. So there exists a path of weight $w$ from $orb(*) \to orb(v') \in \Gamma^{H}$. Further, there is a  path in $\Gamma$ of weight $w$ from $*$ to some vertex $v'' \in orb(v')$ (remember $v''$ is a vertex of $\Gamma$). We will then have $w_V(v'') = w$ and so $v''$ will map to $[w,{orb(v')}]$ under our constructed map. 
\\

If we have an edge $v_1 \to v_2$ in $\Gamma$ (which necessarily has weight $\frac{w_V(v_1)}{w_V(v_2)}$), we have an edge $orb(v_1) \to orb(v_2)$ in $\Gamma^{H}$ of the same weight.  This in turn induces an edge of weight $\frac{w(v_2)}{w(v_1)}$ in $\Gamma^{H}_{tr}$ from $[w_V(v_1), orb(v_1)] \to [w_V(v_2), orb(v_2)]$. Likewise, if we have an edge $[w_V(v_1), orb(v_1)] \to [w_V(v_2), orb(v_2)]$, it implies the existence of an edge of weight $\frac{w_V(v_2)}{w_V(v_1)}$ from $orb(v_1)$ to $orb(v_2)$ in $\Gamma^{H}$. By construction, this means that there exists an edge of weight $\frac{w(v_2)}{w(v_1)}$ from $v_1$ to some vertex $v \in orb(v_2)$. Now $v$ must have weight $w_V(v_1)\frac{w_V(v_2)}{w_V(v_1)} = w_V(v_2)$. So $v = v_2$, since every vertex in an orbit has a distinct weight. Because the map we have constructed is bijective and respects edges, it is a graph isomorphism and $\Gamma_{tr}^{H} \cong \Gamma$.

 \end{proof}

\subsubsection{Recovering a Graph from its Tracial Cover}
\label{sub: recovering  a graph}
A special case of the above process allows us to recover any graph $\Gamma$ from its tracial cover $\Gamma_{tr}$ and a particular action of $T(\Gamma)$ on $\Gamma_{tr}$. To do this, see that the group $T(\Gamma)$ has an action on the graph $\Gamma_{tr}$, where the action of a weight $w(\ell)$ on $\Gamma_{tr}$  sends $[\omega,{v}]$ to $[(\omega \cdot w(\ell)),{v}]$ (in other words, it sends the equivalence class of paths $[p]$ to the equivalence class $[p * \ell]$ for any loop $\ell$ with weight $w(\ell)$, where $p * \ell$ first follows the edges of $\ell$, then of $p$.). This action is well-defined, since precomposing by $\ell$ does not affect the endpoint of the path, and scales all $p' \sim p$ by the same weight.  
\begin{defn}
    If $p_1$ has vertices $v_1, ..., v_k$ and edge weight $\epsilon_1, ..., \epsilon_{k-1}$, then $rev(p_1)$ is the path with vertices $v_k, ..., v_1$ and edges weights $e_{k-1}^{-1}, ..., e_1^{-1}$.
\end{defn}
\begin{lem}
\label{orbit end points}
    Under this action, $[w_1, t_1]$ and $[w_2,t_2]$ share an orbit if and only if $t_1 = t_2$ are paths with the same end point in $\Gamma$.
\end{lem}
\begin{proof}
  First, assume that $[w_1, t_1]$ and $[w_2,t_2]$ share an orbit under the action of $T(\Gamma)$. This means $[w_2, p_2] = \epsilon \cdot [w_1, p_1]$ for some $\epsilon \in T(\Gamma)$ and that there exists a loop $\ell$ with $w(\ell) = \epsilon$ and a path $p_1$ in of weight $w_1$ ending at $t_1$  such that  $p_1 * \ell$ is a path of weight $w_2$ ending at $t_2$. This implies that $t_1 = t_2$. 
  \\

  If we assume $t_1 = t_2$, for any paths $p_1 \in [w_1, t_1]$ and $p_2 \in [w_2, t_2]$,  we can see that $\ell = rev(p_1) * p_2$ is a loop in $\Gamma$ with weight $\epsilon = \frac{w(p_2)}{w(p_1)}$. Now, $p_1 * \ell$ is a path with end point $t_1$. The weight of $p_1 * \ell$ is $ w(p_1) * \epsilon  = w(p_1) *\frac{w(p_2)}{w(p_1)} = w(p_2)$. So we have that $p_1 \cdot \ell = p_2$ and so $\epsilon \cdot [w_1,t_1] = [w_2, t_2]$ and $[w_1, t_1]$, $[w_2, t_2]$ are in the same orbit as desired.
\end{proof}
This quotient on the vertices of $\Gamma_{tr}$ indentifies any two vertices corresponding to paths with the same end point. So the vertices of $\Gamma_{tr}^{T(\Gamma)}$ are in bijection with the vertices of $\Gamma$. Now, we must show that we have recovered the edge structure of $\Gamma$ as well. 

To do this, we first note that we will use the notation $orb([w(p), t(p)])$ to denote a vertex in $\Gamma_{tr}^{T(\Gamma)}$. Of couse, we can simply indentify this with a vertex $v$ in $\Gamma$. Our goal now is to show every edge with source $orb([w(p),t(p)])$ in $\Gamma_{tr}^{T(\Gamma)}$ has a corresponding edge with the same weight and source $v$ in $\Gamma$. 

We have an edge in $\Gamma_{tr}^{T(\Gamma)}$ from $orb([w_1, t_1])$ to $orb([w_2,t_2])$ whenever we have an edge $[w_1, t_1] \to [w_2, t_2]$ in $\Gamma_{tr}$, and the edge inherits it's weight in $\Gamma_{tr}$ (Note: we show that this definition of edges is well-defined in defintion \ref{gamma_H}). This means there is an edge of weight $\frac{w_2}{w_1}$ from $t_1 \to t_2$ in $\Gamma$, by definiton. Likewise, if there is an edge from $v_1 \to v_2$ of weight $w$ in $\Gamma$. We have for any path weight $p$, an edge $[p,v_1] \to [wp, v_2]$ in $\Gamma_{tr}$, and therefore and edge  $orb([p,v_1]) \to orb([wp,v_2])$ (here Lemma \ref{orbit end points} tell us that whichever path we pick, the orbits we get are indentical). So $\Gamma \cong \Gamma^{T(\Gamma)}_{tr}$ as graphs. Thus we can recover any graph $\Gamma$ from $\Gamma_{tr}$ and the action of  $T(\Gamma)$.

\section{Calculating $T_{0}$}
\label{distinguished subset subsection}
Tracial fair and balanced $\delta$-graphs $\Gamma$ correspond to tracial connected $W^{*}$-algebra objects in $TLJ(\delta)$. We would like to use this correspondence to simplify the application of Proposition \ref{t0(m) included in f(m)} to these standard invariants, specifically, to make a connection to the spectral invariant of the algebra object represented by the tracial fair and balanced $\delta$-graph $\Gamma$. We use our defintion of $T(\textbf{G})$ and $T_0(\textbf{M})$ to motivate the following definition.   

\begin{defn}
    For a tracial fair and balanced $\delta$-graph $\Gamma$, we define the set 
    $$T_0(\Gamma) = \{\lambda 
    \in \mathbb{R}^{*}_{+} \,\,\,|\,\,\, \lambda \in T(H) \text{ for } H_{tr} \simeq \Gamma\}$$
\end{defn}

\begin{defn}
For a tracial fair and balanced $\delta$-graph $\Gamma$ with vertex weighting $w_V$, $\alpha \in Aut(\Gamma)$ is an automorphism on $(\Gamma, w_V)$ if it preserves if the ratios $\frac{w_V(t)}{w_V(s)}$ for every edge $s \to t$ in $\Gamma$ (i.e. it preserves the edge weighting). Note that this means that  $w_V(\alpha(v)) = \lambda_\alpha * w_V(v)$ for every $v \in V(\Gamma)$, where $\lambda_\alpha = w_V(\alpha(*))$.  Then, define 
$$W^{\times}(\Gamma) = \{\lambda_\alpha = w_V(\alpha(*)) \,\,\,|\,\,\, \alpha \in Aut((\Gamma, w))\}$$
\end{defn}

\begin{prop}
\label{prop: T_0= Wx}
    For a fair and balanced $\delta$-graph $\Gamma$, $T_0(\Gamma) = W^{\times}(\Gamma)$
\end{prop}
\begin{proof}
    If $\lambda \in T_0(\Gamma)$, by definition $\lambda = w(\ell)$ for some loop $\ell$ in a graph $H$ where $H^{tr} \simeq \Gamma$. Now, see that we can build an automorphism $\alpha_{\ell} \in Aut(\Gamma)$ by sending $[p]$ to $[p * \ell]$ (note that we can use this language since $\Gamma$ is $H^{tr}$). It is easy to see that this automorphism acts multiplicatively and that $w(\alpha_{\ell}(*)) = \lambda$, therefore $\lambda \in W^{\times}(\Gamma)$. 
    \\
    \\
    If $w(\alpha(*)) \in W^{\times}(\Gamma)$, define an action on $\Gamma$ by the homorphism $\mathbb{Z} \to Aut(\Gamma)$ generated by $1 \mapsto \alpha$ . As in section \ref{actions of groups on graphs}, we can construct a graph $\Gamma^{\mathbb{Z}}$ where $\Gamma^{\mathbb{Z}}_{tr} \simeq \Gamma$. We now must show that $w(\alpha(*))$ is the weight of a loop in $\Gamma^{\mathbb{Z}}$. 
    \\
    \\
    To do this, we first note that $*$ and $\alpha(*)$ share an orbit under the action of $\mathbb{Z}$, and therefore they will be mapped to the same vertex in $\Gamma^{\mathbb{Z}}$ by definition. Also, recall that we've constructed edges in $\Gamma^{\mathbb{Z}}$ so that we have for every path in $\Gamma$ with weight $\omega$, a well-defined path in $\Gamma^{\mathbb{Z}}$ with the same weight. Now, any path from $*$ to $\alpha(*)$ in $\Gamma$ has weight $w(\alpha(*))$ by defintion. In $\Gamma^{\mathbb{Z}}$, this path is a loop, and so $w(\alpha(*)) \in T_0(\Gamma)$ as desired. So these sets are equal. 
    \end{proof}
\subsection{Examples of  calculating $T_0(\Gamma)$}
Here we present some examples of how to use propositon \ref{prop: T_0= Wx} to calculate $T_0(\Gamma)$. In examples \ref{single chain actions} and \ref{ex: double chain actions}, we return to the graphs we have used as examples throughout the paper.  We calculate $T_0(\Gamma)$ for any graph $\Gamma$ that can be intrepreted as the Cayley graph of a group with a homomorphism to $\mathbb{R}^{\times}$ in example \ref{ex:cayley graphs} and show in remark \ref{remark: finitely generated R* subgroups} how this implies that every finitely generated subgroup of $\mathbb{R}^{\times}$ appears as $T_0(\Gamma)$ for some $\Gamma$. Finally, we show that $T_0(\Gamma)$ is trivial for both finite fair and balanced $\delta$-graphs (Proposition \ref{finite graph prop}) and for fair and balanced $\delta$-graphs with $|\delta| \approx 2$ (Example \ref{ex: approx 2}). 
\begin{ex}
\label{single chain actions}
Recall the graph $\Gamma_2$ we introduced in example \ref{single chain graph intro}. 
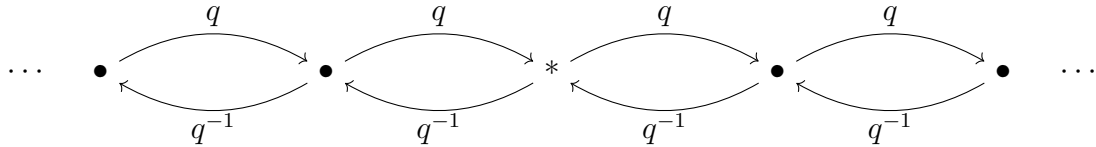
   $$\begin{tikzpicture}
        \node(none) (s) at (0,0){*} ;
        \node(none) (v)at (3,0){$\bullet$};
        \node(none) (w) at (-3,0){$\bullet$};
        \node(none) (x)at (6,0){$\bullet$};
        \node(none) (y) at (-6,0){$\bullet$};
        \node(none) (z) at (-7,0){$\cdots$};
        
        \node(none) (z) at (7,0){$\cdots$};
        \node(none) (1) at (-1.5,0.75){$q$};
        \node(none) (3) at (1.5,0.75){$q $};
         \node(none) (1) at (-4.5,0.75){$q$};
       `\node(none) (3) at (4.5,0.75){$q$};
        \node(none) (1) at (-1.5,-0.75){$q^{-1}$};
        \node(none) (3) at (1.5,-0.75){$q^{-1} $};
         \node(none) (1) at (-4.5,-0.75){$q^{-1}$};
        \node(none) (3) at (4.5,-0.75){$q^{-1} $};

         \draw [->] (w) to [out=30,in=150] (s) ;
        \draw [->] (s) to [out=210,in=330] (w) ;

        \draw [->] (s) to [out=30,in=150] (v) ;
        \draw [->] (v) to [out=210,in=330] (s) ;

        \draw [->] (v) to [out=30,in=150] (x) ;
        \draw [->] (x) to [out=210,in=330] (v) ;

         \draw [->] (y) to [out=30,in=150] (w) ;
        \draw [->] (w) to [out=210,in=330] (y) ;
    \end{tikzpicture}$$
In  example \ref{single chain graph intro}, we showed that $\Gamma_2$ is a tracial fair and balanced $\delta$-graph. The vertex weighting on the graph gives a vertex $k$ steps to the right of $*$ the weight $q^k$ and the vertex $k$ steps to the left of $*$ the weight $q^{-k}$.   Take $H_1 = \langle q \rangle$. We can define an action on $\Gamma_2$ by sending $*$ to the vertex with weight $q$, that is the vertex  to the right of $*$. This choice determines the image of the rest of the vertices in $\Gamma_2$ by definition. This action has one orbit, since $q^j \in H$ for any $j \in \mathbb{Z}$.    So the graph $\Gamma_2^{H_1}$ is trivial (that is, it has only one vertex). 

Now, we can consider the group $H_2 = \langle q^{3} \rangle$. We can define an action of $H_2$ on $\Gamma_2$ by sending $*$ to the unique vertex with weight $q^{3}$. Visually, this action shifts every vertex in $\Gamma_2$ to the right three spots. This action has three distinct orbits. Since as before, $w(h \cdot *) = hw(*) = h$, we cannot have $q, q^{2} \in orb(*)$, since $q, q^{2} \notin H_2$. Also, see that $q \notin orb(q^{2})$, since if $q^{2} = h \cdot q$, then $h = \frac{q^{2}}{q} = q$. Furthermore,  $orb(*), orb(q)$ and $orb(q^{2})$ are the only orbits, since every $q^j = q^{3m}q^\ell$ for any arbitrary $m$ and $\ell \in \{0,1,2\}$. Looking at edges between orbits, we can construct $\Gamma_2^{H_2}$ to be the following graph: 

$$
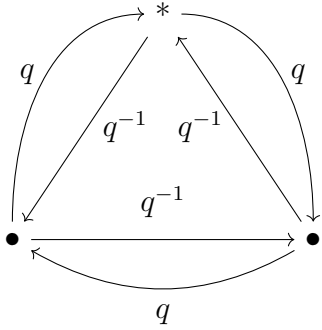
\begin{tikzpicture}
    \node(none) (s) at (0,3){*} ;
    \node(none) (v)at (2,0){$\bullet$};
    \node(none) (w)at (-2,0){$\bullet$};
     \draw [->] (w) to [out=90,in=180] (s) ;
    \draw [->] (s) to  (w) ;

        \draw [->] (s) to [out=0,in=90] (v) ;
        \draw [->] (v) to  (s) ;
         \draw [->] (v) to [out=210,in=330] (w) ;
         \draw[->](w) to  (v);
    \node(none) at (0,0.5){$q^{-1}$} ;
    \node(none) at (-0.5,1.5){$q^{-1}$} ;
      \node(none) at (0.5,1.5){$q^{-1}$} ;
      \node(none) at (0,-1){$q$} ;
    \node(none) at (-1.8,2.2){$q$} ;
      \node(none) at (1.8,2.2){$q$} ;
\end{tikzpicture}$$
For $n \geq 3$, we can tell a similar story for the action of $H_n = \langle q^{n} \rangle$, and $G_2^{H_n}$ will be a cycle with $n$ vertices, with labels of $q$ in one direction and $q^{-1}$ in the other. 

We can use proposition \ref{prop: T_0= Wx} to calculate $T_0(\Gamma_2)$. For each $j \in \mathbb{Z}$, an autmorphism of $\Gamma_2$ sending * to the vertex with weight $q^{j}$. Thus $W^{\times}(\Gamma_2) = \langle q \rangle$, the full set of vertex weights. So $T_0(\Gamma_2) = \langle q \rangle$.
\end{ex}

\begin{ex}
    \label{ex: double chain actions}
    Now, we can also consider the following graph $\Gamma_{a,b}$, which we recall is the tracial cover of the graph $G_1$ presented in example \ref{double chain graph intro}:

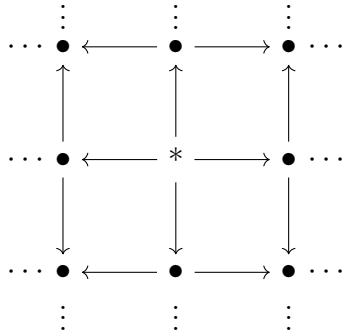
    $$
\begin{tikzpicture}
    \node(none) (s) at (0,0){*} ;
     \node(none) (a) at (1.5, 0){$\bullet$};
     \node(none) (a-1) at (-1.5, 0){$\bullet$};
     \node(none) (b) at (0, 1.5){$\bullet$};
     \node(none) (b-1) at (0, -1.5){$\bullet$};
     \node(none) (ab) at (1.5,1.5){$\bullet$};
     \node(none) (a-1b) at (-1.5,1.5){$\bullet$};
     \node(none) (ab-1) at (1.5,-1.5){$\bullet$};
     \node(none) (a-1b-1) at (-1.5,-1.5){$\bullet$};

      \draw[->] (s) to (a);
     \draw[->] (s) to (a-1);
     \draw[->] (s) to (b);
     \draw[->] (s) to (b-1);
      \draw[->] (a) to (ab);
     \draw[->] (a) to (ab-1);
     \draw[->] (a-1) to (a-1b-1);
    \draw[->](a-1) to (a-1b);
     \draw[->] (b) to (ab);
     \draw[->] (b-1) to (ab-1);
     \draw[->] (b) to (a-1b);
     \draw[->] (b-1) to (a-1b-1);

    \node(none) (aa...) at (2,0) {$\cdots$};
    \node(none) (ab...) at (2,1.5) {$\cdots$};
    \node(none) (ab-1...) at (2,-1.5) {$\cdots$};
    \node(none) (a-1a-1...) at (-2,0) {$\cdots$};
    \node(none) (a-1b...) at (-2,1.5) {$\cdots$};
    \node(none) (a-1b-1...) at (-2,-1.5){$\cdots$};
    \node(none) (bb...) at (0,2) {$\vdots$};
    \node(none) (ba...) at (1.5,2) {$\vdots$};
    \node(none) (ba-1...) at (-1.5,2) {$\vdots$};

    \node(none) (b-1b-1...) at (0,-2) {$\vdots$};
    \node(none) (b-1a...) at (1.5,-2) {$\vdots$};
    \node(none) (b-1a-1...) at (-1.5,-2) {$\vdots$};  
\end{tikzpicture}$$
If we look at  $a^{k}b^{m}$ for integers $k,m$, we can define an action of $H = \langle a^{k}b^{m} \rangle$ on $\Gamma_{a,b}$ by $a^{k}b^{m} \cdot * = v$, where $v$ is the unique vertex in $\Gamma_{a,b}$ with $w(v) = a^{k}b^{m}$. This is a valid action by the symmetry of the graph. We can do this without any restrictions on $k$ and $m$ and these are the only automorphisms that act multiplicatively on $\Gamma_{a,b}$, so $W^{\times}(\Gamma_{a,b}) = \langle a, b \rangle$. Thus $T_0(\Gamma_{a,b}) = \langle a,b \rangle$. 
    
\end{ex}
Notice that in both of these cases, $\Gamma_{a,b}$ and $\Gamma_2$ can be seen as Cayley graphs of infinite groups ($\mathbb{Z}$ and $\mathbb{Z}^{2}$ respectively). We now look at this example in full generality. Note that in the following proposition, we restrict possible weightings to those that assign the same weight to every edge associated with left multiplication by a specific generator. 

\begin{ex}
\label{ex:cayley graphs}
      Let $\Gamma$ be the Cayley graph of a group $G = \langle S \rangle$ with edges representing left multiplication by elements of the given generating set $S$. Any homomorphism $\phi: G \rightarrow \mathbb{R}^{\times}$, induces a fair and balanced $\delta$-graph structure on $\Gamma$. Explicitly, the vertex weighting $w_V$ is defined by $w_V(v) = \phi(g)$ when $v$ is the vertex representing the group element $g$. The edge weighting $u$ is then defined by saying $u(e) = \frac{w_V(t(e))}{w_V(s(e))}$. It is easy to see that if $s(e)$ represent the element $g$ and $t(e)$ represents $sg$ for some $s \in S$, then $u(e) = \phi(s)$. We will show that for  this fair and balanced $\delta$-graph strucutre, $T_0(\Gamma) = \phi(G)$.
      
      Let $V = \{w(v) | v \in V(\Gamma)\}$. By defintion, $W^{\times}(\Gamma) \subseteq V$. We will show the opposite inclusion by constructing an explicit graph autmorphism for each element of $V$.   Choose some weight $w^* \in V$ and let $v^* \in V(\Gamma)$ be some vertex in $\Gamma$ whose weight is $w^*$. Because $\Gamma$ is a Cayley graph for $G$, $v^*$ corresponds to a group element $g^*$. 

    We can define an isomorphism $\rho_{v^*}: V(\Gamma) \to V(\Gamma)$ where $\rho_{v^*}(v) = v'$, where $g_{v'} = g_v * g^*$ (here $g_v$ is the group element corresponding to a vertex $v$). These maps are graph automorphisms, since (if $v' = sv$ for $s \in S$, then $v'g^* = svg^*$ for $g* \in G$, so $\rho_{v^*}$ preserves the edges). We must also show that $\rho_{v^{*}}$ preserves the weighting $w$ on edges and thus the fair and balanced $\delta$-graph structure. 
    
    We will show this by showing it when $g^*$ is a generator $s \in S$. Choose some vertex $v \in V(\Gamma)$ repesenting a group element $g$. Now, $\rho_s$ sends $v$ to $v_s$,the vertex representing $gs$. To show this automorphism respects the weighting, suppose we have an edge $v \to \bar{v}$ with weight $\phi(\bar{s})$, where $\bar{v}$ corresponds to $\bar{s}g$ for some other generator $\bar{s} \in S$. We see that $\rho_s(\bar{v})$ is vertex representing $\bar{s}gs$. So there is an edge $\rho_s(v) \to \rho_s(\bar{v})$ corresponding to multiplication by $\bar{s}$, which must then have weight $\phi(\bar{s})$ by definition. This is enough to show $\rho_s$ respects the weighting, and since any $\rho_{g^{*}}$ is just a composition of $\rho_s$, these maps respect the weighting as well.   
    
    Now, $\rho_g$ sends $*$ to $v^*$, and so $w^* = w(p_{g^*}(*)) \in W^{\times}(\Gamma)$. Therefore, $V \subseteq W^{\times}(\Gamma)$ and so $W^{\times}((\Gamma, w)) = V$. So by proposition \ref{prop: T_0= Wx}, $T_0(\Gamma) = V$. Since $V$ is defined by a homomorphism $\phi$, we can further say that $T_0(\Gamma) = \phi(G)$. 
\end{ex}

\begin{remark} 
    \label{remark: finitely generated R* subgroups}
    The above proposition tells us that any finitely generated subgroup $G$ of $\mathbb{R}^{*}$ will appear as  $T_0(\Gamma)$ for some $\Gamma$ and at least one choice of $\delta$. If $G = \langle g_1, ..., g_k \rangle$  we can choose $\delta = g_1 + g_1^{-1} + ...+ g_k + g_k^{-1}$ and say that $\Gamma$ is the Cayley graph of $G$ as in example \ref{ex:cayley graphs}. We can  choosing a weighting where edges corresponding to multiplication by $g_i$ have weight $g_i$. Using example \ref{ex:cayley graphs}, we see $T_0(\Gamma) = G$, as desired. 
\end{remark}
\begin{ex}
    Suppose $(\Gamma, w)$ is a tracial fair and balanced $\delta$-graph and $G$ is a group that acts simply transitively on $(\Gamma, w)$. For any vertex $v \in V(\Gamma)$, we have $g \in G$ where action by $g$ is an automorphism sending $*$ to $v$. This implies that $w(v) \in W^{\times}((\Gamma, w))$ and so $ T_0((\Gamma, w))$ is the entire set of vertex weights.
\end{ex}
\begin{prop}\label{finite graph prop}
    If $\Gamma$ is a finite tracial fair and balanced $\delta$-graph, then $T_0(\Gamma) = \{1\}$.
\end{prop}
\begin{proof}
    Suppose that $\Gamma$ is a fair and balanced $\delta$-graph where $1 \neq \lambda \in T_0(\Gamma)$. Then there is a fair and balanced $\delta$-graph $\Lambda$ with $\Lambda_{tr} \simeq \Gamma$ so that $\lambda$ is the weight of a loop in $H$. Now, since the loops of $\Lambda$ form a subgroup of $\mathbb{R}^{*}$, there must be an infinte number of non-trivial loop weights. When we construct $\Lambda_{tr}$, we have a vertex for every loop weight $\omega$ (since the vertices of $\Lambda_{tr}$ are indentified by path weight and terminal vertex). So $\Lambda{tr}$ is infinite. 
\end{proof}
\begin{ex}
\label{ex: approx 2}
    The results of  \cite[Section 4]{MR3420332} allow us to examine the invariants we have constructed for small values of $\delta$. First, note that by the arithmetic mean-geometric mean inequality, that if $\delta  = q + q^{-1}$ and $q$ is a positive real number, that $\delta \geq 2$. In particular, \cite[Section 4]{MR3420332} classifies fair and balanced $\delta$-graphs with $|q| \approx 1$, where $|\delta| \approx 2$. Proposition \ref{finite graph prop} tells us, that types $A_n^{(1}, E_g^{(1)}, D'_m$ and $A'_m$ will give us trivial $W^\times(\Gamma)$ for any choice of parameters. Type $A'_\infty$ has no graph automorphisms such that $\alpha(*) \neq *$ and so also has trivial invariant for any parameter. The only non-trivial automorphisms of type $D^{*}_{\infty}$ swap $*$ and $\tilde{*}$, which have the same weighting and therefore the invariant group is trivial. 

    This leaves us with graph of type $A_{\infty, 
\infty}$ to consider. Recall that this is the following graph: 

 $$
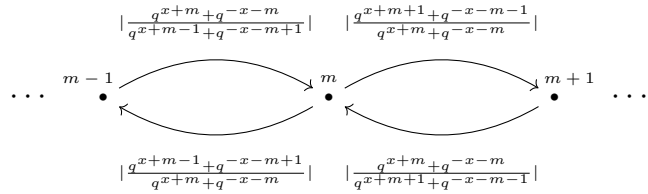
\begin{tikzpicture}
        \node(none) (s) at (0,0.25){\tiny $m$} ;
        \node(none) (s) at (0,0){\tiny$\bullet$} ;
        \node(none) (v)at (3.2,0.25){\tiny $m+1$};
        \node(none) (v)at (3,0){\tiny $\bullet$};
        \node(none) (w) at (-3.2,0.25){\tiny $m-1$};
        \node(none) (w) at (-3,0){\tiny $\bullet$};
        \node(none) (z) at (-4,0){$\cdots$};
        
        \node(none) (z) at (4,0){$\cdots$};
        \node(none) (1) at (-1.5,1){\tiny $|\frac{q^{x+m} +q^{-x-m}}{q^{x+m-1} + q^{-x-m+1}}|$};
        \node(none) (3) at (1.5,1){\tiny$|\frac{q^{x+m + 1} +q^{-x-m - 1}}{q^{x+m} + q^{-x-m}}|$};

        \node(none) (1) at (-1.5,-1){\tiny$|\frac{q^{x+m - 1} +q^{-x-m + 1}}{q^{x+m} + q^{-x-m}}|$};
        \node(none) (3) at (1.5,-01){\tiny$|\frac{q^{x+m} +q^{-x-m}}{q^{x+m + 1} + q^{-x-m - 1}}|$};
        
         \draw [->] (w) to [out=30,in=150] (s) ;
        \draw [->] (s) to [out=210,in=330] (w) ;

        \draw [->] (s) to [out=30,in=150] (v) ;
        \draw [->] (v) to [out=210,in=330] (s) ;

    \end{tikzpicture}$$

We will choose our star vertex as the vertex with $m = 0$. From here, it is relatively simple to see that if $m$ is positive, it's vertex weight will be $\frac{q^{x+m} + q^{-x-m}}{q^{x} + q^{-x}}$ and if $m$ is negative, it's vertex weight will be $\frac{q^{x-m} + q^{-x+m}}{q^{x} + q^{-x}}$. So $W^\times(A_{\infty, \infty}) = \langle \frac{q^{x+m} + q^{-x-m}}{q^{x} + q^{-x}}, \frac{q^{x-m} + q^{-x+m}}{q^{x} + q^{-x}} \rangle$. Since $|q| \approx 1$, $\frac{q^{x+m} + q^{-x-m}}{q^{x} + q^{-x}} \approx \frac{2}{2} = 1$ and $\frac{q^{x-m} + q^{-x+m}}{q^{x} + q^{-x}} \approx \frac{2}{2} = 1$. This means that $T_0(A_{\infty, \infty})$ is approximately $\{1\}$. So when $\delta \approx 2$, there are no non-trivial invariants. 
\end{ex}
---------------------------------------------------
\bibliographystyle{alpha}
{\footnotesize{
\bibliography{bibliography}
}}

\end{document}